\documentclass[a4paper,10pt]{amsart}

\usepackage[T1]{fontenc} \usepackage[utf8]{inputenc}
\usepackage{lmodern}

\usepackage{mathpazo}

\usepackage{enumerate}

\usepackage{amsfonts,amsmath,amstext,amsbsy,amssymb}
\usepackage{amsopn,amsthm,amscd,amsxtra}

\usepackage{latexsym,mathrsfs} 

\usepackage{mathabx} \newcommand{\carr}{\righttoleftarrow}

\usepackage{hyperref}
\RequirePackage[dvipsnames]{xcolor} 
\definecolor{halfgray}{gray}{0.55} 
\definecolor{webgreen}{rgb}{0,0.5,0}
\definecolor{webbrown}{rgb}{.6,0,0} \hypersetup{%
  colorlinks=true, linktocpage=true, pdfstartpage=3,
  pdfstartview=FitV,%
  breaklinks=true, pdfpagemode=UseNone, pageanchor=true,
  pdfpagemode=UseOutlines,%
  plainpages=false, bookmarksnumbered, bookmarksopen=true,
  bookmarksopenlevel=1,%
  hypertexnames=true,
  pdfhighlight=/O,
  urlcolor=webbrown, linkcolor=RoyalBlue,
  citecolor=webgreen, 
  pdftitle={Periodic point free homeomorphisms},%
  pdfauthor={Alejandro Kocsard},%
  pdfsubject={2000 MAthematical Subject Classification: Primary:},%
  pdfkeywords={},%
  pdfcreator={pdfLaTeX},%
  pdfproducer={LaTeX with hyperref}%
}

\usepackage{ae}

\newcommand{\abs}[1]{\left\lvert{#1}\right\rvert}
\newcommand{\norm}[1]{\left\|{#1}\right\|}

\newcommand{\scprod}[2]{\left\langle{#1},{#2}\right\rangle}

\newcommand{\bb}{\mathbb} 
 \newcommand{\ms}{\mathscr}

\newcommand{\R}{\mathbb{R}}\newcommand{\N}{\mathbb{N}}
\newcommand{\Z}{\mathbb{Z}}\newcommand{\Q}{\mathbb{Q}}
\newcommand{\T}{\mathbb{T}}
\newcommand{\A}{\mathbb{A}}

 \newcommand{\ie}{i.e.\ }
\newcommand{\eg}{e.g.\ }

\newtheorem{theorem}{Theorem}[section] \newtheorem*{mainthmA}{Theorem
  A} 
\newtheorem{proposition}[theorem]{Proposition}
\newtheorem{lemma}[theorem]{Lemma}
\newtheorem{corollary}[theorem]{Corollary}

\theoremstyle{definition} \newtheorem{definition}{Definition}[section]

\theoremstyle{remark} \newtheorem{remark}[theorem]{Remark}

\newcommand{\Homeo}[1]{\mathrm{Homeo}_{#1}}

 \newcommand{\dd}{\:\mathrm{d}}

\DeclareMathOperator{\M}{\mathfrak{M}} 
\DeclareMathOperator{\Per}{Per} \DeclareMathOperator{\Fix}{Fix}
\DeclareMathOperator{\SL}{SL} \DeclareMathOperator{\GL}{GL}

\DeclareMathOperator{\diam}{diam}

\newcommand{\cc}{\mathrm{cc}} 

\newcommand{\pr}[1]{\mathrm{pr}_{#1}} 

\newcommand{\eps}{\varepsilon} \newcommand{\U}{\mathscr{U}}
 \newcommand{\Tt}{\mathscr{T}}
\newcommand{\W}{\mathscr{W}}
\DeclareMathOperator{\Fill}{Fill}

\DeclareMathOperator{\supp}{supp}

\title[Irrational rotation factors]{Periodic point free homeomorphisms
  and irrational rotation factors}

\author{Alejandro Kocsard} 

\email{akocsard@id.uff.br}

\thanks{The autor was partially supported by FAPERJ (Brazil) and CNPq
  (Brazil).}

\address{IME - Universidade Federal Fluminense. Rua Prof. Marcos
  Waldemar de Freitas Reis, S/N. Bloco H, $4^\circ$
  andar. 24.210-201, Gragoatá, Niterói, RJ, Brasil}

\date{\today}

\begin{document}

\maketitle

\begin{abstract}
  We provide a complete characterization of periodic point free
  homeomorphisms of the $2$-torus admitting irrational circle
  rotations as topological factors. Given a homeomorphism of the
  $2$-torus without periodic points and exhibiting uniformly bounded
  rotational deviations with respect to a rational direction, we show
  that annularity and the geometry of its non-wandering set are the
  only possible obstructions for the existence of an irrational circle
  rotation as topological factor. Through a very precise study of the
  dynamics of the induced \emph{$\rho$-centralized skew-product}, we
  extend and generalize considerably previous results of Jäger
  \cite{JaegerLinearConsTorus}.
\end{abstract}

\section{Introduction}
\label{sec:intro}

The principal aim of this work consists in giving a complete
characterization of those homeomorphisms of the two-dimensional torus
$\T^2$ that admit an irrational circle rotation as topological
factor. These topological factors, that will be called
\emph{irrational circle factors} for short, are extremely useful in
the dynamical study of these homeomorphisms because they determine an
invariant cyclic order on the system. In fact, the partition of the
torus $\T^2$ given by the fibers of the semi-conjugacy map is
invariant by the dynamics and all the atoms of the partition have a
well-defined rational homological direction, forming an interesting
invariant geometric structure that we called \emph{torus
  pseudo-foliation} in \cite{KocRotDevPerPFree}. Just to mention an
application of these irrational circle factors, and the
pseudo-foliations that they generate, let us say they played a key
role in recent progresses \cite{KocMinimalHomeosNotPseudo,
  KoropeckiPasseggiSambarino, PasseggiSambarinoDevFMConj} on the so
called Franks-Misiurewicz conjecture \cite[Conjecture]{FranksMisiure}.

Regarding necessary conditions for the existence of an irrational
circle factor, the first one we encounter is the so called
\emph{bounded rotational deviations.} In fact, if
$g\colon\T=\R/\Z\carr$ is an arbitrary orientation preserving circle
homeomorphism and $\tilde g\colon\R\carr$ is a lift of $g$, its
\emph{rotation number} $\rho(\tilde g)\in\R$ can be characterized as
the only real number such that
\begin{displaymath}
  \abs{\tilde g^n(x)-x-n\rho(\tilde g)}\leq 1, \quad\forall x\in\R,\
  \forall n\in\Z. 
\end{displaymath}

This property implies (see for instance \cite[Lemma
3.1]{JaegerTalIrratRotFact}) that if $f\colon\T^2\carr$ is an
orientation preserving homeomorphism admitting an irrational circle
factor, then there exist a rational slope vector
$v\in\R^2\setminus\{0\}$ and a positive constant $C=C(f,v)>0$ such
that for any lift $\tilde f\colon\R^2\carr$ of $f$, there is an
irrational number $\rho=\rho(\tilde f,v)\in\R\setminus\Q$ satisfying
\begin{equation}
  \label{eq:v-dev-def}
  \abs{\scprod{\tilde f^n(z)-z}{v}-n\rho}\leq C, \quad\forall
  z\in\R^2,\ \forall n\in\Z.  
\end{equation}

So, it is natural to ask whether estimate \eqref{eq:v-dev-def} is also
sufficient to guarantee the existence of an irrational circle factor.

It is known that this question has a positive answer in some
particular cases. For instance, as a rather straightforward
consequence of classical Gottschalk-Hedlund theorem
\cite{GottschalkHedlundTopDyn}, one can show that any minimal torus
homeomorphism verifying \eqref{eq:v-dev-def} necessarily admits an
irrational circle factor. A much subtler positive result is due to
Jäger~\cite[Theorem C]{JaegerLinearConsTorus}, who showed that any
area-preserving totally irrational pseudo-rotation of $\T^2$ that
exhibits uniformly bounded rotational deviations in any direction (\ie
its rotation set reduces to a single point $\tilde\rho\in\R^2$, where
the corresponding rotation on $\T^2$ induced by $\tilde\rho$ is
minimal, and condition \eqref{eq:v-dev-def} holds for every
$v\in\R^2\setminus\{0\}$) is a topological extension of a totally
irrational (\ie minimal) translation of $\T^2$, where this
semi-conjugacy is constructed as the combination of two
``transversal'' irrational circle factors. For some time these were
the only known positive results. At this point it is interesting to
remark that Wang and Zhang have recently proved in
\cite{WangZhangRigidityPseudoRotNortonSull} the existence of a
$C^\infty$ area-preserving diffeomorphism of $\T^2$ which is a
topological extension of a rigid rotation, but it is not conjugate to
it.

On the other hand, Jäger and Tal have shown in
\cite[\S4]{JaegerTalIrratRotFact} that boundedness of rotational
deviations (\ie estimate \eqref{eq:v-dev-def} holds) does not in
general imply the existence of an irrational circle factor for
homeomorphisms of the closed annulus $\T\times[0,1]$. Consequently,
annularity could be an obstruction for the existence of such factors
for homeomorphisms of $\T^2$ (see \S\ref{sec:rotat-deviations} for
precise definitions and details).

In \S\ref{sec:wand-points-obstruction-Kronecker} we will see that the
existence of wandering points may be also another obstruction for the
existence of irrational circle factors for periodic point free
homeomorphisms satisfying \eqref{eq:v-dev-def}. In fact, we will see
that the topology and geometry of the wandering set play a determinant
role in the existence of irrational circle factors (see
Definition~\ref{def:small-wandering-domains} for details).

The main purpose of this work consists in getting a complete
characterization of possible obstructions for the existence of such
topological factors, showing that annularity and the existence of
``large connected components of wandering points'' are the only
ones. More precisely, our main result is the following

\begin{mainthmA}
  Let $f\in\Homeo{+}(\T^2)$ be an orientation preserving,
  non-eventually annular homeomorphism having small wandering
  domains. Then, $f$ admits an irrational circle rotation as
  topological factor if and only if there exists
  $v\in\R^2\setminus\{0\}$ with rational slope so that $f$ exhibits
  uniformly bounded rotational $v$-deviations, \ie given any lift
  $\tilde f\colon\R^2\carr$, there exist $\rho\in\R\setminus\Q$ and
  $C>0$ such that
  \begin{equation}
    \label{eq:v-bounded-deviation}
    \abs{\langle\tilde f^n(z)-z,v\rangle - n\rho} \leq C, \quad\forall
    z\in\R^2,\ \forall n\in\Z. 
  \end{equation}
\end{mainthmA}

Here \emph{small wandering domains} means that all connected
components of the wandering set are homotopically trivial in $\T^2$
and their diameters are eventually small (see
Definition~\ref{def:small-wandering-domains} for details). It is
important to remark that Jäger and Tal have proved this result under
the area-preserving assumption \cite[Theorem
1.1]{JaegerTalIrratRotFact}.

On the other hand, let us observe that a homeomorphism admitting a
periodic point free factor must be periodic point free itself. So this
imposes some restrictions to the possible isotopy classes of the torus
homeomorphism (see
Proposition~\ref{pro:periodic-point-free-homeos-conj-Hk} for details).

As consequence of Theorem A, we get the following extension of Jäger's
main result of \cite{JaegerLinearConsTorus}:

\begin{corollary}
  \label{cor:Jaeger-generalization}
  If $f\colon\T^2\carr$ is a totally irrational pseudo-rotation
  exhibiting uniformly bounded rotational deviations and having small
  wandering domains (\eg $f$ is non-wandering), then $f$ is a
  topological extension of the corresponding minimal rigid translation
  of $\T^2$.
\end{corollary}

Another consequence of Theorem A is the following
\begin{corollary}
  \label{cor:Dehn-twist-isotopy-class}
  If $f\in\Homeo{}(\T^2)$ is isotopic to the $k$-Dehn twist $I_k$,
  with $k\neq 0$ (see \eqref{eq:Im-def} for definition), and has small
  wandering domains, then $f$ admits an irrational rotation as
  topological factor if and only if given any a lift
  $\tilde f\colon\R^2\carr$ of $f$, there exist $\rho\in\R\setminus\Q$
  and $C>0$ such that
  \begin{equation}
    \label{eq:bounded-vertical-rot-dev}
    \abs{\pr{2}\big(\tilde f^n(z)-z\big)-n\rho} \leq C, \quad\forall
    z\in\R^2,\ \forall n\in\Z,
  \end{equation}
  where $\pr{2} : \R^2\ni (x,y) \mapsto y$.
\end{corollary}

Let us finish this introduction remarking that putting together our
results with some recent ones due to Hauser and
Jäger~\cite{HauserJaegerMonotMaxEq}, one can completely characterize
those periodic point free homeomorphisms of $\T^2$ that admits
non-trivial Kronecker factors. In fact, a Kronecker system is nothing
but a translation on a compact abelian topological group. Its mixed
algebro-topological nature endowed the system with a rich variety of
structures and allows us to combine different techniques (\eg metric
geometry, character theory) to study its dynamical and ergodic
properties. So, it is not surprising at all that the problem of
determining the existence of non-trivial Kronecker factors is a very
important one in topological and differentiable dynamics.

Since any closed subsystem (\ie the restriction of the dynamics to a
closed invariant subset) of a Kronecker one is (conjugate to a)
Kronecker itself, then in general one just considers minimal Kronecker
factors and by \emph{non-trivial Kronecker system} we mean a minimal
one which does not reduce to a point (see \cite[\S
2.6]{TaoPoincareLega} for a very nice exposition about Kronecker
systems).

As the reader may be expecting, any non-trivial Kronecker factor of a
torus homeomorphism is in fact a minimal torus translation and the
semi-conjugacy map cannot increases the dimension. However this result
is not completely straightforward and has just been proven by Hauser
and Jäger \cite{HauserJaegerMonotMaxEq} (see
Theorem~\ref{thm:classification-Kronecker-factors-T2}) in dimension
$2$ and by Edeko~\cite{EdekoEquicontFactorsFlows} in higher
dimensions. So, a non-trivial Kronecker factor of a $2$-torus
homeomorphism can be either an irrational circle rotation or a totally
irrational, \ie minimal, $2$-torus rotation; but the second case can
be performed as the combination of two ``transversal'' circle factors
(see \S\ref{sec:proofs-corollaries} for details), so this is the
reason we just focus on the existence of irrational circle factors.

\subsection{Strategy of the proofs for the main results}
\label{sec:strategy-proof-main-results}

In \S\ref{sec:proofs-corollaries} we prove Corollaries
\ref{cor:Jaeger-generalization} and \ref{cor:Dehn-twist-isotopy-class}
assuming Theorem A.

The proof of the main result of the paper, Theorem A, is performed in
\S\S \ref{sec:rho-centralized-skew} and \ref{sec:proof-thm-A}. In
\S\ref{sec:rho-centralized-skew} we introduce the
\emph{$\rho$-centralized skew-product}, which is a slight modification
of similar constructions we performed in \cite{KocRotDevPerPFree,
  KocMinimalHomeosNotPseudo}, but adapted to a more general
context. In that section we we make a very detailed study of the
dynamics of these $\rho$-centralized skew-products, showing the
following results: in Proposition~\ref{pro:F-orbit-bound-vert-dir} we
prove that an orbit of the original homeomorphism exhibits bounded
rotational deviations if and only if the corresponding orbits of the
$\rho$-centralized skew-product are bounded. Then, in
Theorem~\ref{thm:Omega-f-vs-Omega-F} we characterize the non-wandering
set of the $\rho$-centralized skew-product showing, for instance, that
this skew-product is non-wandering if and only if the original
homeomorphism if non-wandering. Then, in
Theorem~\ref{thm:separating-sets-orbits-of-open-sets} we study the
topology of open bounded connected invariant sets of the
$\rho$-centralized skew-product showing that, under the hypotheses of
Theorem A, their complements have two unbounded connected components
on each fiber. Finally, in \S~\ref{sec:proof-thm-A} we finish the
proof of Theorem A, showing that the boundaries of some bounded
connected open sets which are invariant under the $\rho$-centralized
skew-product can be used to construct the fibers of the semi-conjugacy
of the original homeomorphism.

\subsection*{Acknowledgments}
\label{sec:acknowledgements}

The author would like to thank Tobias Jäger and Andrés Koropecki for
several useful discussions and to the anonymous referees for their
criticisms that helped to improved the readability of the paper.

\section{Preliminaries and notations}
\label{sec:preliminaries-notations}

\subsection{General topological dynamics}
\label{sec:gener-topol-dynamics}

Throughout this article, $(M,d)$ will denote an arbitrary complete
metric space. The (open) ball of radius $r>0$ centered at $x\in M$
will be denoted by $B_r(x)$. Given any $A\subset M$, we write
$\partial_XA$ for the boundary of $A$ in $X$, and $\overline{A}$ for
its closure; its diameter is given by
\begin{displaymath}
  \diam A:=\sup\{d(x,y) : x,y\in A\}. 
\end{displaymath}
We say $A$ is bounded when $\diam A$ is finite.

If $A$ is connected, we write $\cc(M,A)$ for the connected component
of $M$ containing $A$. As usual, we write $\pi_0(M)$ to denote the set
of connected components of the space $M$.

The group of self-homeomorphisms of $M$ will be denoted by
$\Homeo{}(M)$. We shall write $\Homeo0(M)$ for the subgroup of
homeomorphisms of $M$ which are homotopic to the identity.

Given any $f\in\Homeo{}(M)$, we define its \emph{support} as the
closed set
\begin{equation}
  \label{eq:supp-f-def}
  \supp f :=\overline{\left\{x\in M : f(x)\neq x\right\}}. 
\end{equation}

When $M$ and $N$ are two topological spaces, we say that
$f\in\Homeo{}(M)$ is a \emph{topological extension} of
$g\in\Homeo{}(N)$ when there exists a surjective continuous map
$h\colon M\to N$ such that $h\circ f = g\circ h$. In such a case, we
say that $g$ is a \emph{topological factor} of $f$ and the map $h$ is
called a \emph{semi-conjugacy}. A \emph{fiber} of the semi-conjugacy
is the pre-image by $h$ of any point of $N$. The set of maps that are
a semi-conjugacy between $f$ and $g$ will be denoted by
$\mathcal{SC}(f,g)$, \ie it is defined by
\begin{equation}
  \label{eq:semi-conj-maps}
  \mathcal{SC}(f,g):=\left\{h'\in C^0(M,N) : h'(M)=N,\ h'\circ f =
    g\circ h'\right\}. 
\end{equation}

Whenever $M_1,M_2,\ldots, M_n$ are arbitrary sets, we shall use the
generic notation
$\pr{i}\colon M_1\times M_2\times\ldots\times M_n\to M_i$ to denote
the $i^\mathrm{th}$-coordinate projection map. Similarly, when $A$ is
a subset of $M_1\times M_2\times\ldots\times M_n$ and $x\in M_1$ is an
arbitrary point, we write $A_x$ to denote the \emph{``the fiber of $A$
  over $x$''} given by
\begin{equation}
  \label{eq:fiber-A-over-x}
  A_x:=\{(x_2,\ldots,x_n)\in M_2\times\ldots\times M_n :
  (x,x_2,\ldots,x_n)\in A\}. 
\end{equation}

\subsubsection{Recurrent and non-wandering points}
\label{sec:recurr-non-wandering}

Let $f\in\Homeo{}(M)$ be any homeomorphism. A point $x\in M$ is said
to be \emph{recurrent} when there exists an increasing sequence of
positive integers $(n_j)_{j\geq 1}$ such that $f^{n_j}(x)\to x$, as
$j\to\infty$. An open subset $U\subset M$ is said to be a
\emph{wandering set} (for $f$) when $f^n(U)\cap U=\emptyset$, for
every $n\in\Z\setminus\{0\}$. A point $x\in M$ is called
\emph{non-wandering} when no neighborhood of $x$ is wandering. The
\emph{non-wandering set} of $f$, \ie the set of non-wandering points,
shall be denoted by $\Omega(f)$; its complement, called the
\emph{wandering set} of $f$, will be denoted by
$\W(f):=M\setminus\Omega(f)$.

When $M$ is locally connected, each connected component of $\W(f)$ is
called a \emph{wandering domain.} Finally, we say $f$ is
\emph{non-wandering} when $\Omega(f)=M$. We will need a new notion
which is stronger than non-wandering:

\begin{definition}
  \label{def:non-wandering-recurrent}
  A homeomorphism $f\colon M\carr$ is said to be
  \emph{$\Omega$-recurrent} when for every open subset $U\subset M$
  satisfying $U\cap\Omega(f)\neq\emptyset$, there is $n\geq 1$ such
  that
  \begin{displaymath}
    U\cap f^{-n}(U)\cap\Omega(f)\neq\emptyset. 
  \end{displaymath}
\end{definition}

Our main interest in this notion is due to the following elementary
\begin{lemma}
  \label{lem:non-wandering-vs-recurrent-pts}
  Let $M$ be a complete metric space and $f\colon M\carr$ be an
  $\Omega$-recurrent homeomorphism.

  Then, recurrent points are dense within the set $\Omega(f)$.
\end{lemma}

\begin{proof}
  Let $x$ be any non-wandering point and $U$ an arbitrary open
  neighborhood of $x$. Without loss of generality, we can assume $U$
  is bounded in $M$. Then we will inductively defined a sequence of
  nested open sets $\{U_k\}_{k\geq 0}$ and an increasing sequence of
  positive integers $\{n_k\}_{k\geq 0}$ as follows: first, let us
  define $U_0:=U$. Then, assuming $k\geq 1$ and $U_{k-1}$ has already
  been defined, we write
  \begin{equation}
    \label{eq:nk-time-def}
    n_{k-1}:=\min\left\{n\geq 1 : U_{k-1}\cap
      f^{-n}(U_{k-1})\cap\Omega(f)\neq\emptyset\right\}.
  \end{equation}
  Then, we define $U_k$ as any open set satisfying the following
  properties:
  \begin{gather*}
    \overline{U_k}\subset U_{k-1}\cap f^{-n_{k-1}}(U_{k-1}), \\
    U_k\cap\Omega(f)\neq\emptyset,\\
    \diam(U_k)<\diam(U_{k-1})/2.
  \end{gather*}

  Observe that, since $f$ is $\Omega$-recurrent, $U_{k-1}$ is open and
  $\Omega(f)$ is closed, such an open set $U_k$ does exists and the
  natural number given by \eqref{eq:nk-time-def} is well defined.
  
  By our hypothesis about the diameter of the sets $U_k$, it follows
  there is a unique point $y$ such that
  \begin{displaymath}
    \{y\}=\bigcap_{k=0}^\infty U_k = \bigcap_{k=0}^\infty
    \overline{U_k},
  \end{displaymath}
  and we claim $y$ is a recurrent point. In fact,
  $y\in U_{k+1}\subset U_k\cap f^{-n_k}(U_k)$, for every $k\geq 0$,
  and hence
  \begin{displaymath}
    d(f^{n_k}(y),y)\leq\diam U_k\to 0,\quad\text{as }  k\to\infty.  
  \end{displaymath}
  On the other hand, by the very same definition it holds
  $n_k\geq n_{k-1}$, for very $k\geq 1$. This implies $y$ is a
  recurrent point and $y\in U$. So, recurrent points are dense among
  non-wandering points.
\end{proof}

\subsubsection{Minimal systems}
\label{sec:minimal-systems}

Let $M$ be a compact metric space, $f\colon M\carr$ an arbitrary
homeomorphism and $K\subset M$ be an $f$-invariant nonempty compact
set. We say that $K$ is a \emph{minimal set} (for $f$) when $K$ and
the empty set are the only $f$-invariant compact subsets of $K$. When
$M$ itself is a minimal set, we say that $f$ is a \emph{minimal
  system} or a \emph{minimal homeomorphism.}

\subsubsection{Proximality relations}
\label{sec:proximality-relat}

If $f\colon M\carr$ is a homeomorphism of the complete metric space
$(M,d)$, we say that two points $x,y\in M$ are \emph{$f$-proximal}
when
\begin{equation}
  \label{eq:proximal-def}
  \inf\left\{d\big(f^n(x),f^n(y)\big) : n\geq 1\right\} = 0.
\end{equation}

Analogously we can define the notion of $f^{-1}$-proximality. Notice
that, in general, these two notions do not coincide.

\subsection{Orbits of open sets}
\label{sec:orbits-open-sets}

Let $M$ be an arbitrary locally connected complete metric space and
$f\colon M\carr$ be an arbitrary homeomorphism. Given any nonempty
connected open set $V$, following Koropecki and
Tal~\cite{KoroTalStricTor} we define
\begin{equation}
  \label{eq:U-eps-def}
  \U_f(V):=\cc\bigg(\bigcup_{n\in\Z} f^n\big(V\big),
  V\bigg), 
\end{equation}
where $\cc(\cdot,\cdot)$ denotes the connected component as defined at
the beginning of \S\ref{sec:gener-topol-dynamics}. Notice that there
exists $N\geq 1$ such that $f^N\big(\U_f(V)\big)=\U_f(V)$ if and only
if $V$ is not a wandering set.

\subsection{Kronecker factors}
\label{sec:kronecker-factors}

Let $(G,+)$ denote an abelian group. For each $a\in G$, let us
consider the translation $T_a\in\Homeo{}(G)$ given by
$T_a : G\ni g\mapsto g+a$. When $(G,+)$ is an abelian compact group,
any such map is called a \emph{Kronecker system.}

\begin{remark}
  \label{rem:translations-especial-notation}
  For the sake of simplicity of notations, given $E\subset G$ and
  $a\in G$ sometimes we shall write $E+a$ to denote $T_a(E)$.
\end{remark}

All orbit closures of a Kronecker system are (topologically conjugate
to) a Kronecker subsystem themselves. So, there is no significant loss
of generality just considering minimal Kronecker systems.

We will say $T_a\colon G\carr$ is a \emph{non-trivial Kronecker
  system} when $T_a$ is minimal and $G$ is not just a singleton.

By classical arguments one can easily show that any compact abelian
group $G$ admits a compatible distance $d_G$ which is invariant by any
translation, \ie every Kronecker system is an isometry of
$(G,d_G)$. In particular this implies that any Kronecker system is
equicontinuous (\ie the family of its iterates is equicontinuous) when
$G$ is endowed with an arbitrary compatible metric.

Reciprocally, it can be shown that any minimal equicontinuous
homeomorphism of an arbitrary compact metric space is topologically
conjugate to a minimal Kronecker system (see for instance
\cite[Proposition 2.6.7]{TaoPoincareLega} for details).

Given a homeomorphism of a compact metric space $f\colon M\carr$, a
\emph{Kronecker factor} of $f$ will be a Kronecker system
$T_a\colon G\carr$ which is a topological factor of $f$, \ie there
exists a surjective continuous map $h\colon M\to G$ satisfying
$h\circ f=T_a\circ h$.

We will need the following
\begin{definition}
  \label{def:Kronecker-equivalence}
  Let $M$ be a compact metric space and $f\colon M\carr$ be a
  homeomorphism. We say that two points $x,y\in M$ are \emph{Kronecker
    equivalent} for $f$ if their images coincides under any
  semi-conjugacy of a Kronecker factor, \ie for every Kronecker factor
  $T_a\colon G\carr$ of $f$ and any $h\in\mathcal{SC}(f,T_a)$, it
  holds $h(x)=h(y)$ (see \eqref{eq:semi-conj-maps} for this notation).

  On the other hand, we say that $x$ and $y$ are \emph{Kronecker
    separated} if and only if $h(x)\neq h(y)$, for every non-trivial
  Kronecker factor $T_a\colon G\carr$ and every
  $h\in\mathcal{SC}(f,T_a)$.
\end{definition}

\begin{remark}
  \label{rem:proximality-vs-Kronecker-equivalence}
  Notice that, if $f\colon M\carr$ is as in
  Definition~\ref{def:Kronecker-equivalence}, then two points
  $x,y\in M$ which are either $f$-proximal or $f^{-1}$-proximal are
  necessarily Kronecker equivalent.
\end{remark}

\subsection{Tori and torus homeomorphisms}
\label{sec:tori-torus-homeos}

We will always consider $\R^d$ endowed with the Euclidean inner
product $\scprod{v}{w}:=\sum_{i=1}^dv_iw_i$ and the induced Euclidean
norm $\norm{v}:=\sqrt{\scprod{v}{v}}$.

As usual, the $d$-torus $\R^d/\Z^d$ will be denoted by $\T^d$ and we
write $\pi\colon\R^d\to\T^d$ for the canonical quotient
projection. Notice we are using the same letter $\pi$ to the denote
the torus universal covering map independently of its dimension.

We shall always assume $\T^d$ endowed with the distance function
$d_{\T^d}$ given by
\begin{equation}
  \label{eq:d-Td-dist-def}
  d_{\T^d}(x,y):=\min\left\{\norm{\tilde x-\tilde y} : \tilde
    x\in\pi^{-1}(x),\ \tilde y\in\pi^{-1}(y)\right\}, \quad\forall
  x,y\in\T^d.
\end{equation}

An open set $D\subset\R^d$ is called a \emph{fundamental domain} for
the covering $\pi\colon\R^d\to\T^d$ when $\pi$ is injective on $D$ and
$\pi\big(\overline{D}\big)=\T^d$.

As a particular case of Kronecker systems we have torus translations
$T_\alpha\colon\T^d\carr$, with $\alpha\in\T^d$. By some abuse of
notation and for the sake of simplicity, if $\alpha\in\R^d$, we shall
just write $T_\alpha$ to denote $T_{\pi(\alpha)}$.

A vector $\alpha\in\R^d$ is called \emph{rational} when
$\alpha\in\Q^d$; otherwise, it is called \emph{irrational}. Moreover,
it is called \emph{totally irrational} when the translation
$T_\alpha\colon\T^d\carr$ is minimal.

It is well known that given any $f\in\Homeo{}(\T^d)$, there exists a
unique matrix $A_f\in\GL(d,\Z)$ such that the map
\begin{equation}
  \label{eq:Delta-f-definition}
  \Delta_{\tilde f}:=\tilde f-A_f\colon\R^d\to\R^d
\end{equation}
is $\Z^d$-periodic, for any lift $\tilde f\colon\R^d\carr$ of $f$; and
hence the function $\Delta_{\tilde f}$ can be considered as an element
of $C^0(\T^d,\R^d)$. The matrix $A_f$ is nothing but a matrix
representation of the action induced by $f$ on the first homology
group of $\T^d$, and one can easily check that two homeomorphisms
$f,g\in\Homeo{}(\T^d)$ are homotopic if and only if $A_f=A_g$.

Finally, given any $k\in\Z$, we write
\begin{equation}
  \label{eq:Im-def}
  I_k:=
  \begin{pmatrix}
    1 & k \\
    0 & 1
  \end{pmatrix}\in\SL(2,\R),
\end{equation}
and define
\begin{equation}
  \label{eq:Homeok-Td}
  \Homeo{k}(\T^2):=\left\{f\in\Homeo{}(\T^2) : A_f=I_k\right\}.
\end{equation}

\subsubsection{Rotation set and rotation vectors}
\label{sec:rot-set-rot-vect}

We write $\Homeo{0}(\T^d)$ to denote the group of torus homeomorphisms
which are homotopic to the identity.

After Misiurewicz and Ziemian \cite{MisiurewiczZiemian}, given any
lift $\tilde f\colon\R^d\carr$ of a homeomorphism $f\in\Homeo0(\T^d)$
one defines the \emph{rotation set of $\tilde f$} by
\begin{equation}
  \label{eq:rho-set-def}
  \rho(\tilde f)=\bigcap_{m\geq 0}
  \overline{\bigcup_{n\geq m}\bigg\{\frac{\tilde f^n(z)-z}{n} :
    z\in\R^d\bigg\}} = \bigcap_{m\geq 0}
  \overline{\bigcup_{n\geq m}\left\{\frac{\Delta_{\tilde f^n}(z)}{n} :
      z\in\T^d\right\}},
\end{equation}
where $\Delta_{\tilde f}$ is the displacement function given by
\eqref{eq:Delta-f-definition}.  It can be easily shown that the
rotation set $\rho(\tilde f)$ is always nonempty, compact and
connected.

For $d=1$, by classical Poincaré theory~\cite{Poincare1880memoire} of
circle homeomorphisms we know that $\rho(\tilde f)$ is always a
singleton and its class modulo $\Z$ depends just on $f$ and not on the
chosen lift. So, in such a case one can define
$\rho(f):=\pi(\rho(\tilde f))\in\T$.

In higher dimensions, in general, the rotation set is not
singleton. So, we will say that $f\in\Homeo{0}(\T^d)$ is a
\emph{pseudo-rotation} whenever $\rho(\tilde f)$ is a singleton for
some, and hence any, lift $\tilde f\colon\R^d\carr$ of $f$; and $f$ is
said to be a \emph{totally irrational pseudo-rotation} when
$\rho(\tilde f)$ reduces to a point which is a totally irrational one.
 
On the other hand, given an $f$-invariant Borel probability measure
$\mu$ one can define the \emph{$\mu$-rotation vector of $\tilde f$} by
\begin{equation}
  \label{eq:rho-mu-def}
  \rho_\mu(\tilde f):=\int_{\T^d}\Delta_{\tilde f}\dd\mu, 
\end{equation}
where $\Delta_{\tilde f}$ is the displacement function given by
\eqref{eq:Delta-f-definition}. By classical convexity arguments and
Birkhoff ergodic theorem it can be easily checked that
\begin{displaymath}
  \mathrm{Conv}\big(\rho(\tilde f)\big) =\left\{\rho_\mu(\tilde f) :
    \mu\in\M(f)\right\}, 
\end{displaymath}
where $\M(f)$ is the space of $f$-invariant Borel probability measures
and $\mathrm{Conv}(\cdot)$ denotes the convex hull operator.

For $d=2$, Misiurewicz and Ziemian~\cite[Theorem
3.4]{MisiurewiczZiemian} showed that the rotation set is convex
indeed, so it coincides with the set of rotation vectors of measures.

On the other hand, when $f\in\Homeo{k}(\T^2)$ with $k\neq 0$, one
cannot define the rotation set as above, but at least one can define
the \emph{vertical rotation set} as in \cite{DoeffRotMeasHomeos,
  DoeffMisiurewiczShearRotNum, ZanataPropVertRotInt}, given by
\begin{equation}
  \label{eq:rot-vertical-def}
  \rho_{\rm{V}}(\tilde f)=\bigcap_{m\geq 0}
  \overline{\bigcup_{n\geq m}\bigg\{\frac{\pr{2}\big(\tilde
      f^n(z)-z\big)}{n}: z\in\R^2\bigg\}}\subset\R,
\end{equation}
where $\tilde f\colon\R^2\carr$ is any lift of $f$. Analogously, for
any $\mu\in\M(f)$ one defines its \emph{vertical rotation number} by
\begin{displaymath}
  \rho_{\mu,\rm{V}}(\tilde f):= \int_{\T^2}\pr{2}\circ\Delta_{\tilde
    f}\dd\mu. 
\end{displaymath}

\subsubsection{Dimension two}
\label{sec:dimension-two}

Let us fix some especial notations for the two-dimensional
case. Given any $v=(a,b)\in\R^2$, we define $v^\perp:=(-b,a)$. A
vector $v\in\R^2\setminus\{0\}$ is said to have \emph{rational slope}
when there is $\lambda\in\R\setminus\{0\}$ such that
$\lambda v\in\Z^2$; and \emph{irrational slope} otherwise.

We write $\A:=\T\times\R$ for the open annulus. We consider the
covering maps $\tilde\pi\colon\R^2\to\A$ and $\hat\pi\colon\A\to\T^2$
given by the natural quotient projections
\begin{align}
  \label{eq:tilde-pi-def}
  \tilde\pi : \R^2\ni (\tilde x,\tilde y) &\mapsto \big(\tilde
                                            x+\Z,\tilde 
                                            y\big)\in\A,\\
  \label{eq:hat-pi-def}
  \hat\pi: \A\ni (x,\tilde y) &\mapsto \big(x,\tilde
                                y+\Z\big)\in\T^2. 
\end{align}
Observe that $\pi=\hat\pi\circ\tilde\pi\colon\R^2\to\T^2$.

We will always considered $\A$ endowed with the distance function
$d_\A$ given by
\begin{displaymath}
  d_\A(\hat z_0,\hat z_1) := \min\left\{\norm{\tilde z_1-\tilde z_0} :
    \tilde z_i\in\tilde\pi^{-1}(\hat z_i), \ i=0,1\right\}, 
\end{displaymath}
for every $\hat z_i\in\A$ and $i=0,1$.

For each $s\in\R$, we define $\hat T_s\in\Homeo0(\A)$ by
$\hat T_s : (x,\tilde y) \mapsto (x,\tilde y +s)$.

If $S$ is any surface, by a \emph{topological disk} in $S$ we mean an
open subset of $S$ which is homeomorphic to the unit disc
$\{z\in\R^2 : \norm{z}<1\}$. Analogously, a \emph{topological annulus}
is an open subset of $S$ which is homeomorphic to $\A$.

If $U\subset\T^2$ is an open nonempty set, let us consider the group
homomorphism induced by the inclusion on first homology groups
$i\colon H_1(U,\Z)\to H_1(\T^2,\Z)$. The set $U$ is said to be
\emph{inessential} when $i=0$, and \emph{essential} otherwise. The set
$U$ is called \emph{annular} when the image of $i$ has rank $1$ and
\emph{fully essential} when its rank is equal to $2$, \ie when $i$ is
surjective.

Analogously, a nonempty open set $U\subset\A$ is said to be
\emph{inessential} when the inclusion morphism
$i\colon H_1(U,\Z)\to H_1(\A,\Z)$ is identically zero; and it is said
to be \emph{annular} otherwise.

An arbitrary subset $E$ of either $\T^2$ or $\A$ is said to be
\emph{inessential} when there exists an inessential open set $U$
containing $E$; otherwise, is said to be \emph{essential.}

A compact connected subset $C\subset\T^2$ is called an \emph{annular
  continuum} when its complement $\T^2\setminus C$ is homeomorphic to
$\A$. On the other hand, we say that a set $C\subset\A$ is an
\emph{essential annular continuum} when it is compact, connected and
$\A\setminus C$ has exactly two connected components, which will be
denoted by $U^+(C)$ and $U^-(C)$, and they are characterized by the
fact that there is a real number $K>0$ such that
$\T\times (K,+\infty)\subset U^+(C)$ and
$\T\times (-\infty,-K)\subset U^-(C)$.

The \emph{filling} of an inessential open subset $U$ of $\T^2$ is
defined as the union of $U$ with all the inessential connected
components of its complement, and will be denoted by
$\Fill(U)$. Notice that $\Fill(U)$ is an open inessential subset
itself.

A connected open set $U\subset\T^2$ ($U\subset\A$, respectively) is
said to be \emph{lift-bounded} when every connected component of
$\pi^{-1}(U)$ ($\tilde\pi^{-1}(U)$, respectively) is bounded in
$\R^2$; and it is called \emph{lift-unbounded} otherwise. Notice that
every lift-bounded set is necessarily inessential, but there do exist
open inessential lift-unbounded subsets of $\T^2$ and $\A$.

The main reason why we are interested in the spaces $\Homeo{k}(\T^2)$
of homeomorphisms given by \eqref{eq:Homeok-Td} is given by the
following

\begin{proposition}
  \label{pro:periodic-point-free-homeos-conj-Hk}
  If $f\colon\T^2\carr$ is an orientation preserving fixed point free
  homeomorphism, then there exists a unique $k\in\Z$ such that $f$ is
  topologically conjugate to an element of $\Homeo{k}(\T^2)$.
\end{proposition}

\begin{proof}
  This is a straightforward consequence of Lefschetz fixed point
  theorem. In fact, if $f$ is orientation preserving and has no fixed
  point, then it must hold
  \begin{displaymath}
    0=L(f)=\sum_{i=0}^2(-1)^i\mathrm{tr}(f_{\star,i}\colon
    H_i(\T^2,\Q)\carr) = 2 -  \mathrm{tr}(f_{\star,1}\colon
    H_1(\T^2,\Q)\carr),
  \end{displaymath}
  where $L(f)$ denotes the Lefschetz number of $f$. But the matrix
  $A_f\in\SL(2,\Z)$ is a representative of $f_{\star,1}$, and hence,
  $1$ is the only eigenvalue of the matrix $A_f$. If $A_f=I_0$, \ie is
  $A_f$ is the identity, we are done; if not, $\ker(A_f-I_0)$ is a
  one-dimensional space, and since the matrix $A_f-I_0$ has integer
  coefficients, this space is generated by a vector $(a,c)\in\Z^2$,
  with $\gcd(a,c)=1$. So, there exists $b,d\in\Z$ such that $ad-bc=1$
  and this implies the matrix
  \begin{displaymath}
    B=
    \begin{pmatrix}
      a & b\\
      c & d
    \end{pmatrix}\in\SL(2,\Z),
  \end{displaymath}
  and $B^{-1}A_fB=I_k$, for some $k\in\Z$. Then, the linear map $B$
  induces (\ie is the lift to $\R^2$ of) a Lie group automorphism
  $\bar B\colon\T^2\carr$, and one can easily check that
  $\bar B^{-1}\circ f\circ \bar B\in\Homeo{k}(\T^2)$.
\end{proof}

Let us recall a classical result about fixed point free plane
homeomorphisms due to Brouwer:

\begin{theorem}[Brouwer's translation theorem, see
  \cite{FathiOrbCloBrouwer}]
  \label{thm:Brouwer}
  Let $f\colon\R^2\carr$ be an orientation preserving homeomorphism
  such that $\Fix(f)=\emptyset$. Then, every $x\in\R^2$ is wandering
  for $f$, \ie $\Omega(f)=\emptyset$.
\end{theorem}

If $f\in\Homeo{}(\T^2)$ and $x\in\Omega(f)$, following Koropecki and
Tal~\cite{KoroTalStricTor} we say that $x$ is \emph{inessential} when
there exists $\eps>0$ such that the open set
$\U_f\big(B_\eps(x)\big)$, given by \eqref{eq:U-eps-def}, is
inessential; otherwise is said to be an \emph{essential point.}
Moreover, $x$ is said to be a \emph{fully essential point} when
$\U_f\big(B_\eps(x)\big)$ is fully essential, for every $\eps>0$.

We have the following results for periodic point free homeomorphisms:

\begin{proposition}
  \label{pro:ppf-homeos-no-inessential-points}
  If $f\colon\T^2\carr$ is a periodic point free homeomorphism and
  $x\in\Omega(f)$, then $x$ is an essential point.
\end{proposition}

\begin{proof}
  This is an easy consequence of Theorem~\ref{thm:Brouwer}. In fact,
  if $x\in\Omega(f)$ and there is an open neighborhood $V$ of $x$ such
  that $\U_f(V)$ is inessential, hence the filling
  $\Fill\big(\U_f(V)\big)$, as defined in \ref{sec:dimension-two}, is
  $f^N$-invariant, for some $N\in\N$. Notice that
  $\Fill\big(\U_f(V)\big)$ is homeomorphic to $\R^2$ and
  $f^N|_{\Fill\big(\U_f(V)\big)}\colon\Fill\big(\U_f(V)\big)\carr$
  exhibits a non-wandering point (because $x\in\Omega(f)\cap
  V$). Thus, by Brouwer's translation theorem
  (Theorem~\ref{thm:Brouwer}), $f^{2N}$ has a fixed point in
  $\Fill\big(\U_f(V)\big)$, contradicting the fact that $f$ is
  periodic point free.
\end{proof}

\begin{proposition}
  \label{pro:lift-bounded-iness-wandering-domains-ppf-homeos}
  If $f\colon\T^2\carr$ is a periodic point free homeomorphism and
  $W\subset\T^2$ is a lift-bounded wandering domain (\ie $W$ is a
  connected component of $\W(f)$), then
  \begin{displaymath}
    f^n(W)\cap W=\emptyset,\quad\forall n\in\Z\setminus\{0\}. 
  \end{displaymath}
\end{proposition}

\begin{proof}
  Let us suppose there exists $n\in\Z\setminus\{0\}$ such that
  $f^n(W)\cap W\neq\emptyset$. Since $W$ is a connected component of
  an $f$-invariant set, this implies that $f^n(W)=W$. So, if
  $\tilde W$ is a connected component of $\pi^{-1}(W)\subset\R^2$,
  then there exists a homeomorphism $G\colon\R^2\carr$ which is a lift
  of $f^n$ and such that $G(\tilde W)=\tilde W$. Since $\tilde W$ is
  bounded in $\R^2$, this implies that $\Omega(G)$ is nonempty and,
  invoking Brouwer's translation theorem we conclude that $G^2$ has a
  fixed point, contradicting the fact that $f$ is periodic point free.
\end{proof}

Motivated by this result, we propose the following

\begin{definition}
  \label{def:small-wandering-domains}
  We say that a homeomorphism $f\colon\T^2\carr$ exhibits \emph{small
    wandering domains} when every connected component of the wandering
  set $\W(f)$ is lift-bounded and, given any $\delta>0$, there exist
  just finitely many connected components with diameter larger that
  $\delta$, \ie the following set
  \begin{displaymath}
    \left\{W\in\pi_0\big(\W(f)\big) :
      \diam(W)\geq\delta\right\} 
  \end{displaymath}
  is finite, for every $\delta>0$.
\end{definition}

From this geometric property of the wandering set we get the following
dynamical consequence:
\begin{proposition}
  \label{pro:small-wand-domains-implies-Omega-recurr}
  If $f\in\Homeo{}(\T^2)$ is periodic point free and exhibits small
  wandering domains, then it is $\Omega$-recurrent (see
  Definition~\ref{def:non-wandering-recurrent}).
\end{proposition}

\begin{proof}
  Reasoning by contradiction, let us suppose there exists an open set
  $U\subset\T^2$ such that $U\cap\Omega(f)\neq\emptyset$, but
  \begin{equation}
    \label{eq:U-no-non-wandering-recurrent}
    U\cap f^n(U)\cap\Omega(f)=\emptyset, \quad\forall
    n\in\Z\setminus\{0\}. 
  \end{equation}

  Let $z$ be an arbitrary point of $U\cap\Omega(f)$ and $\eps$ be a
  positive number such that the ball of radius $\eps$ centered at $z$
  satisfies $\overline{B_\eps(z)}\subset U$. Taking into account that
  $f$ exhibits small wandering domain, we know that
  \begin{displaymath}
    \mathscr{W}_U(z,\eps):=\left\{D\in\pi_0\big(\W(f)\big) :
      B_\eps(z)\cap D\neq\emptyset,\ \overline{D}\cap
      \big(\T^2\setminus U\big)\neq\emptyset\right\}   
  \end{displaymath}
  is a finite set.

  Then, let us consider the set
  \begin{equation}
    \label{eq:returning-times-Beps-z}
    \tau:=\left\{n\in\Z\setminus\{0\} :
      f^n\big(B_\eps(z)\big)\cap B_\eps(z)\neq\emptyset\right\}. 
  \end{equation}
  Notice that since $z$ is a non-wandering point, the set $\tau$ is
  nonempty, and moreover, is infinite.

  On the other hand, by \eqref{eq:U-no-non-wandering-recurrent} we
  know that
  \begin{equation}
    \label{eq:returning-times-Beps-z-on-W}
    f^{n}\big(B_\eps(z)\big)\cap B_\eps(z)\subset\W(f), \quad\forall
    n\in\tau. 
  \end{equation}
   
  Then, if $D$ is a connected component of the wandering set $\W(f)$
  such that $\overline D\subset U$, its boundary
  $\partial D\subset\Omega(f)$ and hence, by
  \eqref{eq:U-no-non-wandering-recurrent}, it holds
  $f^n(\overline{D})\cap U=\emptyset$, for every
  $n\in\Z\setminus\{0\}$.

  This implies that
  \begin{equation}
    \label{eq:returning-times-Beps-z-on-WUzeps}
    f^n\big(B_\eps(z)\big)\cap B_\eps(z)\subset \bigcup_{D\in
      \mathscr{W}_{U}(z,\eps)} D, \quad\forall n\in\tau.
  \end{equation}

  However, since $f$ is periodic point free, by
  Proposition~\ref{pro:lift-bounded-iness-wandering-domains-ppf-homeos}
  we know that every connected component of $\W(f)$ is indeed a
  wandering set for $f$, and we know that the set
  $\mathscr{W}_U(z,\eps)$ is finite. So, putting together
  \eqref{eq:returning-times-Beps-z},
  \eqref{eq:returning-times-Beps-z-on-W} and
  \eqref{eq:returning-times-Beps-z-on-WUzeps} we conclude the set
  $\tau$ is finite, contradicting the fact that $z$ was non-wandering.
\end{proof}

\subsection{Kronecker factors of homeomorphisms of $\T^2$}
\label{sec:kronecker-factors-T2}

As we have already mentioned in \S\ref{sec:kronecker-factors}, any
minimal equicontinuous homeomorphism of a compact metric space is
topologically conjugate to a Kronecker system. So, invoking a recent
result due to Hauser and Jäger \cite[Theorem
B]{HauserJaegerMonotMaxEq}, we have the following

\begin{theorem}
  \label{thm:classification-Kronecker-factors-T2}
  If $f\in\Homeo{}(\T^2)$ and $T_a\colon G\carr$ is a minimal
  Kronecker factor of $f$, then the group $G$ is either equal to
  $\T^2$, $\T$ or the trivial space $\{*\}$.
\end{theorem}

According to the classification given by
Theorem~\ref{thm:classification-Kronecker-factors-T2}, any
homeomorphism $f\in\Homeo{}(\T^2)$ admitting a minimal Kronecker
factor is a topological extension of an irrational circle rotation,
\ie $f$ admits an irrational circle rotation as topological factor,
that will be just call \emph{irrational circle factors} for short.

The following simple lemma imposes some well-known restrictions for
the existence of irrational circle factors:

\begin{lemma}
  \label{lem:non-triv-Kron-factor-implies-BRD}
  Let $f\colon\T^2\carr$ be an orientation preserving homeomorphism
  admitting an irrational circle factor and $\tilde f\colon\R^2\carr$
  be a lift of $f$. Then there exist $\rho\in\R\setminus\Q$,
  $v\in\bb{S}^1$ with rational slope and $C>0$ such that
  \begin{displaymath}
    \abs{\scprod{\tilde f^n(z)-z}{v} -n\rho}\leq C,\quad\forall
    z\in\R^2,\ \forall n\in\Z.
  \end{displaymath}
\end{lemma}

\begin{proof}
  See for instance \cite[Lemma 3.1]{JaegerTalIrratRotFact}.
\end{proof}

The direction of the vector $v^\perp$, where $v$ is as in
Lemma~\ref{lem:non-triv-Kron-factor-implies-BRD}, is called the
\emph{homological direction} of the circle factor. Notice that, when
$f$ is not homotopic to the identity, the homological direction of a
circle factor is unique.


\subsection{Rotational deviations}
\label{sec:rotat-deviations}

Given an orientation preserving homeomorphism $f\colon\T^2\carr$, we
say that $f$ exhibits \emph{uniformly bounded $v$-deviations,} for
some $v\in\bb{S}^1$, whenever there exists a constant $C>0$ such that
given any lift $\tilde f\colon\R^2\carr$, there is $\rho\in\R$
satisfying
\begin{equation}
  \label{eq:unif-bound-v-dev-def}
  \abs{\scprod{\tilde f^n(z)-z}{v}-n\rho}\leq C,\quad\forall
  z\in\R^2,\ \forall n\in\Z.
\end{equation}
Observe that if the vector $v$ has irrational slope, then the
homeomorphism $f$ must be isotopic to the identity. However, in this
work, as a consequence of
Lemma~\ref{lem:non-triv-Kron-factor-implies-BRD}, we are mainly
concerned with the case where $v$ has rational slope. Moreover, since
we are dealing with periodic point free homoemorphisms, by
Proposition~\ref{pro:periodic-point-free-homeos-conj-Hk} we know that,
modulo conjugacy with a torus automorphism, we can assume our
homeomorphism $f\in\Homeo{k}(\T^2)$, for some $k\in\Z$, the vector
$v=(0,1)$ and the number $\rho$ is the only element of of the vertical
rotation set defined in \S\ref{sec:rot-set-rot-vect}. In such a case
we will say that $f$ exhibits \emph{uniformly bounded vertical
  deviations}.

As a particular case of our previous definition, a homeomorphism
$f\colon\T^2\carr$ is said to be \emph{annular} when there exist a
lift $\tilde f\colon\R^2\carr$ of $f$, a vector $v\in\bb{S}^1$ with
rational slope and a constant $C>0$ such that
\begin{displaymath}
  \abs{\scprod{\tilde f^n(z)-z}{v}}\leq C, \quad\forall z\in\R^2,\
  \forall n\in\Z.
\end{displaymath}
More generally, we say $f$ is \emph{eventually annular} when there
exists $k\in\N$ such that $f^k$ is annular.

We will need the following improvement of
Proposition~\ref{pro:ppf-homeos-no-inessential-points}:

\begin{proposition}
  \label{pro:ppf-nonannular-homeos-fully-essential-points}
  If $f\in\Homeo{}(\T^2)$ is periodic point free and is not eventually
  annular, then every non-wandering point is fully essential for $f$.
\end{proposition}

\begin{proof}
  This is consequence of
  Proposition~\ref{pro:ppf-homeos-no-inessential-points} and the
  following simple remark: if $x\in\Omega(f)$ is essential but is not
  fully essential, then there is an open connected neighborhood $V$ of
  $x$ such that $\U_f(V)$ is an annular set and there is $k\in\N$ such
  that $f^k\big(\U_f(V)\big) = \U_f(V)$. This clearly implies that
  $f^k$ is an annular homeomorphism.
\end{proof}

On the other hand, we will say that the homeomorphism
$f\in\Homeo{k}(\T^2)$ exhibits \emph{unbounded horizontal deviations}
when
\begin{equation}
  \label{eq:horizontal-unbounded-deviations}
  \sup\left\{\abs{\pr{1}\big(\tilde f^n(z)-z\big) -\pr{1}\big(\tilde
      f^n(w)-w\big)} : w,z\in\R^2,\ n\in\Z\right\}=\infty. 
\end{equation}
Notice that any element of $\Homeo{k}(\T^2)$ exhibits unbounded
horizontal deviations, when $k\neq 0$.

We have the following simple

\begin{lemma}
  \label{lem:unbounded-hor-deviations-positive-and-negative}
  Let $f\in\Homeo{k}(\T^2)$ and $\tilde f\colon\R^2\carr$ be a lift of
  $f$. Then the following properties are equivalent:
  \begin{enumerate}
  \item estimate \eqref{eq:horizontal-unbounded-deviations} holds;
  \item we have
    \begin{equation}
      \label{eq:horizontal-unbounded-deviations-positive-times}
      \sup\left\{\abs{\pr{1}\big(\tilde f^n(z)-z\big) -\pr{1}\big(\tilde
          f^n(w)-w\big)} : w,z\in\R^2,\ n\in\N\right\}=\infty;
    \end{equation}
  \item it holds
    \begin{equation}
      \label{eq:horizontal-unbounded-deviations-negative-times}
      \sup\left\{\abs{\pr{1}\big(\tilde f^{-n}(z)-z\big)
          -\pr{1}\big(\tilde f^{-n}(w)-w\big)} : w,z\in\R^2,\
        n\in\N\right\}=\infty. 
    \end{equation}
  \end{enumerate}
\end{lemma}

\begin{proof}
  First observe that given any $z\in\R^2$, we have
  \begin{displaymath}
    \tilde f^n(z) - \tilde f^n\big(z+(0,1)\big) = I_k^n(0,1) = (kn,1),
    \quad\forall n\in\Z.
  \end{displaymath}
  This implies that, when $k\neq 0$,
  \eqref{eq:horizontal-unbounded-deviations},
  \eqref{eq:horizontal-unbounded-deviations-positive-times} and
  \eqref{eq:horizontal-unbounded-deviations-negative-times} hold
  simultaneously, for every $f\in\Homeo{k}(\T^2)$ and any lift
  $\tilde f\colon\R^2\carr$.

  So let us just consider the case $k=0$. It is clear that condition
  \eqref{eq:horizontal-unbounded-deviations} holds if and only if
  either \eqref{eq:horizontal-unbounded-deviations-positive-times} or
  \eqref{eq:horizontal-unbounded-deviations-negative-times} holds. So,
  it just remains to prove that
  \eqref{eq:horizontal-unbounded-deviations-positive-times} and
  \eqref{eq:horizontal-unbounded-deviations-negative-times} are
  equivalent.

  To prove this, it is enough to observe that, if we write $\tilde
  f=I_0+\Delta_{\tilde f}$, then it holds $\tilde
  f^{-1}=I-\Delta_{\tilde f}\circ \tilde f$. So we have
  \begin{displaymath}
    \Delta_{\tilde f^{-n}}= -\Delta_{\tilde f^n}\circ\tilde f^n,
    \quad\forall n\geq 1,
  \end{displaymath}
  and this clearly shows that
  \eqref{eq:horizontal-unbounded-deviations-positive-times} holds if
  and only if
  \eqref{eq:horizontal-unbounded-deviations-negative-times} holds.  
\end{proof}

Finally, a homeomorphism $f\in\Homeo0(\T^2)$ is said to be a
\emph{pseudo-rotation with uniformly bounded rotational deviations}
when $f$ is a pseudo-rotation and there is a constant $C>0$ such that
\begin{equation}
  \label{eq:pseudo-rot-unif-bound-dev}
  \norm{\tilde f^n(z) -z -n\bar\rho}\leq C,\quad\forall z\in\R^2,\
  \forall n\in\Z,
\end{equation}
where $\tilde f\colon\R^2\carr$ is a lift of $f$ and its rotation set
satisfies $\rho(\tilde f)=\{\bar\rho\}$.

Let us finish this section with two lemmas about torus homeomorphisms
exhibiting unbounded deviations in certain direction:

\begin{lemma}
  \label{lem:eventually-large-non-wandering-open-sets}
  Let $f\in\Homeo{k}(\T^2)$, $\tilde f\colon\R^2\carr$ be a lift of
  $f$ and suppose $f$ exhibits unbounded horizontal rotational
  deviations, \ie condition \eqref{eq:horizontal-unbounded-deviations}
  holds.  If $x\in\Omega(f)$ is a fully essential point for $f$ and
  $V$ is a connected neighborhood of $x$, then it holds
  \begin{equation}
    \label{eq:large-deformation-nw-sets}
    \sup_{n\in\N}\diam\Big(\pr{1}\big(\tilde f^n(\tilde V)\big)\Big) = 
    \infty, 
  \end{equation}
  for every connected component $\tilde V$ of $\pi^{-1}(V)$.
\end{lemma}

\begin{proof}
  Without loss of generality we can assume $V$ is open and
  lift-bounded. Since $x$ is fully essential, there are two continuous
  simple closed curves $\alpha_1,\alpha_2\colon\T\to\T^2$ that
  generates the first homology group $H_1(\T^2,\Z)$ and such that
  their images are contained in $\U_f(V)$. This implies there exists
  $C>0$ such that
  \begin{displaymath}
    \diam D<C, \quad\forall
    D\in\pi_0\big(\R^2\setminus\pi^{-1}(\alpha_1\cup\alpha_2)\big), 
  \end{displaymath}
  where, by some abuse of notation, we are just writing $\alpha_i$ to
  denote the image of the curve $\alpha_i$, and $\pi_0(\cdot)$ denotes
  the set of connected components of the corresponding space.
  
  Then, let us consider the open set
  \begin{equation}
    \label{eq:Q-fundamental-domain-def}
    Q:=(0,1)^2\cup \bigcup\left\{D\in
      \pi_0\big(\R^2\setminus\pi^{-1}(\alpha_1\cup\alpha_2)\big),\
      D\cap (0,1)^2\neq\emptyset\right\}.
  \end{equation}
  Notice that, by the previous estimate, $Q$ is bounded in $\R^2$ and
  its boundary is contained in $\pi^{-1}(\alpha_1\cup\alpha_2)$.

  So, if $\tilde V$ is an arbitrary connected component of
  $\pi^{-1}(V)$, for each $z\in\partial Q$ there exist $n_z\in\Z$ and
  $\boldsymbol{p}_z\in\Z^2$ such that
  \begin{displaymath}
    z\in \tilde f^{n_z}(\tilde V) + \boldsymbol{p}_z,
  \end{displaymath}
  where we are using the notation introduced in
  Remark~\ref{rem:translations-especial-notation}

  By compactness of $\partial Q$, there exist finitely many points
  $z_1,\ldots,z_\ell\in\partial Q$ such that
  \begin{equation}
    \label{eq:boundary-Q-cover-translate-f-V}
    \partial Q \subset \bigcup_{i=1}^\ell \tilde f^{n_{z_i}}(\tilde V)
    + \boldsymbol{p}_{z_i}.
  \end{equation}
  This implies
  \begin{displaymath}
    \tilde f^n(\partial Q) \subset\bigcup_{i=1}^\ell \tilde
    f^{n+n_{z_i}}(\tilde V) + I_k^n \boldsymbol{p}_{z_i}, \quad\forall
    n\in\Z, 
  \end{displaymath}
  where $I_k$ denotes the integer matrix given by \eqref{eq:Im-def}.

  Then, since $Q$ contains a fundamental domain and
  \eqref{eq:horizontal-unbounded-deviations} holds, by
  \eqref{eq:horizontal-unbounded-deviations-positive-times} of
  Lemma~\ref{lem:unbounded-hor-deviations-positive-and-negative} we
  conclude that
  \begin{displaymath}
    \sup_{n\in\N} \diam \Big(\pr{1}\big(\tilde f^n(\partial
    Q)\big)\Big) = \infty. 
  \end{displaymath}
  Now, for each $n\in\Z$, $\tilde f^n(\partial Q)$ is covered by
  $\ell$ (integer translations of) images of $\tilde V$, so estimate
  \eqref{eq:large-deformation-nw-sets} must hold.
\end{proof}

\begin{lemma}
  \label{lem:essential-annular-images-open-sets}
  Let $f\in\Homeo{k}(\T^2)$, $\tilde f\colon\R^2\carr$ be a lift of
  $f$ and suppose $f$ exhibits uniformly bounded vertical rotational
  deviations and unbounded horizontal rotational deviations. Let
  $x\in\Omega(f)$ be a fully essential point for $f$, $\tilde{V}$ a
  neighborhood of some point $\tilde x\in\pi^{-1}(x)\subset\R^2$, and
  $N_0$ be a natural number. Then there exists $m\in\N$, such that for
  every $j\in\{m,m+1,\ldots, m+N_0\}$, there are $p_j,p_j',q_j\in\Z$,
  with $p_j\neq p_j'$, such that
  \begin{displaymath}
    \tilde f^j(\tilde V)\cap T_{(p_j,q_j)}(\tilde V)\neq\emptyset,\
    \text{and}\quad \tilde f^j(\tilde V)\cap T_{(p_j',q_j)}(\tilde 
    V)\neq\emptyset, 
  \end{displaymath}
  for every $j\in\{m,m+1,\ldots,m+N_0\}$.
\end{lemma}

Before proving Lemma~\ref{lem:essential-annular-images-open-sets}, we
need the following elementary combinatorial result:

\begin{lemma}
  \label{lem:no-gap-integers}
  Let $A=\{n_1,n_2,\ldots, n_\ell\}$ be an arbitrary nonempty set of
  integer numbers and $N_0$ be a natural number. Then, there exists an
  integer $M_0\geq N_0$ such that for every $m'\in\N$ and any function
  $\xi\colon\{m',m'+1,\ldots,m'+M_0\}\to A$, there is $m\in\Z$ such
  that
  \begin{displaymath}
    \{m,m+1,\ldots,m+N_0\} \subset \left\{ j-\xi(j) :
      j\in\{m',m'+1,\ldots,m'+M_0\}\right\}.
  \end{displaymath}
\end{lemma}

\begin{proof}[Proof of Lemma~\ref{lem:no-gap-integers}]
  Let us suppose the elements of $A$ are ordered in the following way:
  $n_1\leq n_2\leq \ldots\leq n_\ell$. Then, let us define
  \begin{align*}
    M_0&:= N_0+n_\ell-n_1, \\
    m&:=m'-n_1.
  \end{align*}

  For each $1\leq i\leq \ell$, consider the set
  $B_i:=\{j-n_i : j\in\{m',\ldots,m'+M_0\}\}$.

  Then, the lemma easily follows from the following simple remark:
  \begin{align*}
    B_i\cap B_k \supset B_1\cap B_\ell
    &= \{n\in\Z : m'-n_1\leq n\leq m'+M_0-n_\ell\}  \\
    &=\{m,m+1,\ldots,m+M_0\},
  \end{align*}
  for every $1\leq i\leq k\leq\ell$.
\end{proof}

\begin{proof}
  Without loss of generality we can assume $\tilde V=B_r(\tilde x)$,
  with $0<r<1/4$. Then, $V:=\pi(\tilde V)=B_r(x)\subset\T^2$, where
  $x=\pi(\tilde x)\in\Omega(f)$ is a fully essential point for $f$.

  Repeating the argument we used in the proof of
  Lemma~\ref{lem:eventually-large-non-wandering-open-sets}, there are
  two continuous simple closed curves
  $\alpha_1,\alpha_2\colon\T\to\T^2$ that generate the first homology
  group $H_1(\T^2,\Z)$ and such that their images are contained in
  $\U_f(V)$.

  Then let us consider the fundamental domain $Q$ given by
  \eqref{eq:Q-fundamental-domain-def}. Let
  $z_1,\ldots,z_\ell\in\partial Q$, $n_{z_1},\ldots,n_{z_\ell}\in\Z$
  and $\boldsymbol{p}_{z_1},\ldots,\boldsymbol{p}_{z_\ell}\in\Z^2$
  such that condition \eqref{eq:boundary-Q-cover-translate-f-V}
  holds. Let $M_0$ be the natural number given by
  Lemma~\ref{lem:no-gap-integers} associated to the natural numbers
  $n_{z_1}, \ldots, n_{z_\ell}$ and $N_0$.

  Let us consider the set
  \begin{equation}
    \label{eq:Q-hat-definition}
    \check{Q} := Q \cup \bigcup_{i=1}^\ell \tilde f^{n_{z_i}}(\tilde V)
    + \boldsymbol{p}_{z_i}. 
  \end{equation}

  Now, since $f$ exhibits bounded vertical rotational deviations, we
  know that its vertical rotation set given by
  \eqref{eq:rot-vertical-def} is a singleton and there exists a real
  constants $C_0>0$ such that
  \begin{equation}
    \label{eq:C0-bound-for-vertical-deviation}
    \abs{\pr{2}\big(\tilde f^n(z)-z\big) - n\rho}\leq C_0, \quad\forall
    z\in\R^2,\ \forall n\in\Z,
  \end{equation}
  where $\rho_{\mathrm{V}}(\tilde f)=\{\rho\}$.
  
  This implies that
  \begin{equation}
    \label{eq:Sn-strip-def}
    \tilde f^n(\tilde V)\subset  S_n:=\left\{ z\in\R^2 :
      \abs{\pr{2}(z-\tilde x) - n\rho}\leq C_0+1/4\right\},
    \quad\forall n\in\Z.
  \end{equation}

  Then, consider the set of integer numbers
  \begin{equation}
    \label{eq:En-vertical-integers}
    E_n:=\{q\in\Z : \exists p\in\Z,\ T_{(p,q)}(\check Q)\cap
    S_n\neq\emptyset\}, \quad\forall n\in\Z.
  \end{equation}
  Since $\check{Q}$ has finite diameter, we know that
  \begin{equation}
    \label{eq:En-cardinality-estimate}
    \sharp E_n\leq 2C_0+\diam \check{Q} +4, \quad\forall n\in\Z,
  \end{equation}
  where $\sharp(\cdot)$ denotes the cardinality of the set, and on the
  other hand it clearly holds
  \begin{equation}
    \label{eq:Q-disjoint-integer-translates}
    \check{Q}\cap T_{\boldsymbol{p}}(\check{Q})=\emptyset,
    \quad\forall \boldsymbol{p}\in\Z^2,\ \text{with }
    \norm{\boldsymbol{p}}>\diam\check{Q}.
  \end{equation}

  Taking into account \eqref{eq:En-cardinality-estimate},
  \eqref{eq:Q-disjoint-integer-translates}, we consider an integer
  number $L\in\N$ such that
  \begin{equation}
    \label{eq:L-estimate-from-below}
    L> \ell(2C_0+\diam\check{Q} + 4) + 1. 
  \end{equation}
  
  Now, by Lemma~\ref{lem:eventually-large-non-wandering-open-sets}
  \begin{displaymath}
    \sup_{n\in\N} \diam\Big(\pr{1}\big(\tilde f^n(\tilde V)\big)\Big)
    = \infty. 
  \end{displaymath}
  So, we can find a positive integer $m'$ such that for every
  $j\in\{m',m'+1,\ldots,m'+M_0\}$, there are
  $\boldsymbol{p}^{(j)}_1,\ldots,\boldsymbol{p}^{(j)}_L\in\Z^2$ with
  $\norm{\boldsymbol{p}^{(j)}_r-\boldsymbol{p}^{(j)}_s}>
  \diam\check{Q}$, for every $1\leq r<s\leq L$, and satisfying
  \begin{displaymath}
    \tilde f^j\big(\tilde V\big)\cap
    T_{\boldsymbol{p}^{(j)}_r}(Q)\neq\emptyset, \quad\forall
    j\in\{m',\ldots,m'+M_0\},\ \forall r\in\{1,\ldots,L\}, 
  \end{displaymath}
  where
  \begin{displaymath}
    \partial T_{\boldsymbol{p}^{(j)}_r}(Q)\subset \bigcup_{i=1}^\ell
    \tilde f^{n_{z_i}} (\tilde V) + \boldsymbol{p}_{z_i} +
    \boldsymbol{p}^{(j)} _r.
  \end{displaymath}
  This means that for each $j\in\{m',\ldots,m'+M_0\}$ and every
  $r\in\{1,\ldots,L\}$, there exists $z_{j,r}\in\{z_1,\ldots,z_\ell\}$
  such that
  \begin{equation}
    \label{eq:fj-inter-fznjr}
    \tilde f^j(\tilde V)\cap \Big(\tilde f^{n_{z_{j,r}}}(\tilde
    V) + \boldsymbol{p}_{z_{j,r}} +
    \boldsymbol{p}_r^{(j)}\Big)\neq\emptyset.
  \end{equation}

  Observe that, by \eqref{eq:Q-disjoint-integer-translates}, the sets
  $T_{\boldsymbol{p}^{(j)}_1}(\check{Q}),
  T_{\boldsymbol{p}^{(j)}_2}(\check{Q}), \dots,
  T_{\boldsymbol{p}^{(j)}_L}(\check{Q})$ are two-by-two disjoint which
  implies that
  \begin{equation}
    \label{eq:fnzjr-disjoint}
    \Big(\tilde f^{n_{z_{j,r}}}(\tilde
    V) + \boldsymbol{p}_{z_{j,r}} + \boldsymbol{p}_r^{(j)}\Big) \cap
    \Big(\tilde f^{n_{z_{j,s}}}(\tilde V) + \boldsymbol{p}_{z_{j,s}} +
    \boldsymbol{p}_s^{(j)}\Big) =\emptyset,
  \end{equation}
  for all $j\in\{m',\ldots,m'+M_0\}$, and any $1\leq r<s\leq L$.

  Then, by \eqref{eq:Sn-strip-def} and \eqref{eq:En-vertical-integers}
  we yield
  \begin{displaymath}
    \pr{2}\Big(\boldsymbol{p}^{(j)}_r\Big)\in E_j, \quad\forall
    r\in\{1,\ldots,L\},
  \end{displaymath}
  and any $m'\leq j\leq m'+M_0$.  Therefore, putting together this
  last relation with \eqref{eq:En-cardinality-estimate},
  \eqref{eq:fj-inter-fznjr}, \eqref{eq:fnzjr-disjoint} and
  \eqref{eq:L-estimate-from-below}, we conclude that for each
  $j\in\{m',\ldots,m'+M_0\}$, there are $1\leq r<s\leq L$ such that
  $z_{j,r}=z_{j,s}$ and
  $\pr{2}\big(\boldsymbol{p}^{(j)}_r\big)=
  \pr{2}\big(\boldsymbol{p}^{(j)}_s\big)$. So, for each $j$ we get
  \begin{align*}
    \tilde f^{j-n_{z_{j,r}}}(\tilde V)\cap \Big(\tilde V +
    I_k^{-n_{z_{j,r}}}\big(\boldsymbol{p}_{z_{j,r}}-
    \boldsymbol{p}_r^{(j)}\big)\Big) &\neq\emptyset, \\
    \tilde f^{j-n_{z_{j,s}}}(\tilde V)\cap \Big(\tilde V +
    I_k^{-n_{z_{j,s}}}\big(\boldsymbol{p}_{z_{j,s}}-
    \boldsymbol{p}_s^{(j)}\big)\Big) &\neq\emptyset,
  \end{align*}
  where $I_k$ is the matrix given by \eqref{eq:Im-def},
  $\boldsymbol{p}_{z_{j,r}}-\boldsymbol{p}_r^{(j)}\neq
  \boldsymbol{p}_{z_{j,s}}-\boldsymbol{p}_s^{(j)}$ but their second
  coordinates coincide, \ie
  \begin{displaymath}
    \begin{split}
      \pr{2}\big(\boldsymbol{p}_{z_{j,r}}- \boldsymbol{p}_r^{(j)}\big)
      &= \pr{2}\big(\boldsymbol{p}_{z_{j,s}}-
      \boldsymbol{p}_s^{(j)}\big) \\
      &= \pr{2}\Big(I_k^{-n_{j,r}}\big(\boldsymbol{p}_{z_{j,r}}-
      \boldsymbol{p}_r^{(j)}\big)\Big) = \pr{2}\Big(I_k^{-n_{j,s}}
      \big(\boldsymbol{p}_{z_{j,s}}- \boldsymbol{p}_s^{(j)}\big)\Big).
    \end{split}
  \end{displaymath}
  Then, the existence of the natural number $m$ and the conclusion of
  the lemma itself follow by Lemma~\ref{lem:no-gap-integers}.
\end{proof}

\section{Wandering points as an obstruction to irrational circle
  factors}
\label{sec:wand-points-obstruction-Kronecker}

The main purpose of this section consists in describing the
construction of three examples of totally irrational pseudo-rotations
with uniformly bounded rotational deviations but which do not admit
irrational circle rotations as topological factors.

In these three examples the geometry of the wandering sets plays a
fundamental role, showing that the \emph{small wandering domains}
hypothesis is fundamental and sharp in Theorem A.

First we shall perform a general construction, which is a slight
modification of classical suspensions, that will be used in
\S\S\ref{sec:unbounded-iness-wandering-set} and
\ref{sec:fully-essent-wandering}.

\subsection{Suspending circle homeomorphisms}
\label{sec:susp-circle-homeos}

Given arbitrary homeomorphisms $g_1,g_2\in\Homeo0(\T)$ and a lift
$\tilde g_1\colon\R\carr$ of $g_1$, we will construct the
\emph{``time-$\tilde g_1$ of the suspension flow of $g_2$,''} a
homeomorphism $g\in\Homeo0(\T^2)$ which is defined as follows.

First, consider the equivalence relation $\sim$ on $\R\times\T$ given
by
\begin{equation}
  \label{eq:suspension-relation-def}
  (s,x)\sim (s',x') \iff s'-s\in\Z\ \text{and } x=g_2^{s'-s}(x').
\end{equation}
Since $g_2$ is homotopic to the identity, the quotient space
$(\R\times\T)/\!\sim$ is indeed homeomorphic to $\T^2$, and we shall
just identify them without any further reference.

As usual, we define the \emph{suspension flow}
$\Phi\colon \R\times\T^2\to\T^2$ by
\begin{equation}
  \label{eq:suspension-flow-of-g2}
  \Phi^t[s,x]:=[t+s,x], \quad\forall (s,x)\in\R\times\T,\
  \forall t\in\R, 
\end{equation}
where $[s,x]$ denotes the equivalence class of the point $(s,x)$.

Then, we define the homeomorphism $g\colon\T^2\carr$ as the
\emph{``time-$\tilde g_1$''} of $\Phi$. More precisely, we write
\begin{equation}
  \label{eq:g-time-g1-of-suspension}
  g[t,x]:=[\tilde g_1(t),x], \quad\forall (t,x)\in\R\times\T, 
\end{equation}
where $\tilde g_1\colon\R\carr$ is a lift of $g_1$.

In order to verify that $g$ is indeed well defined, it is enough to
notice that $\tilde g_1$ commutes with any integer translation on
$\R$.

Then we have the following
\begin{proposition}
  \label{pro:g-homeo-general-properties}
  The homeomorphism $g\colon\T^2\carr$ we constructed above exhibits
  the following properties:
  \begin{enumerate}[(i)]
  \item\label{cond:g-skew-prod-struc} $g$ is a skew-product, \ie it
    leaves invariant the ``vertical circle foliation'' given by the
    $\big\{\{x\}\times\T : x\in\T\big\}$;
  \item\label{cond:g-tot-irr-pseudo} $g$ is a pseudo-rotation and its
    rotation vector is totally irrational if and only if the numbers
    $1,\rho(\tilde g_1),\rho(\tilde g_1)\rho(\tilde g_2)$ are linearly
    independent over $\Q$, for any lift $\tilde g_2\colon\R\carr$ of
    $g_2$;
  \item\label{cond:g-unif-bound-rot-dev} $g$ exhibits uniformly
    bounded rotational deviations (\ie condition
    \eqref{eq:pseudo-rot-unif-bound-dev} holds);
  \item\label{cond:g-inness-wand-set} if $g_2$ is a Denjoy
    homeomorphism (\ie it is periodic point free and
    $\W(g_2)\neq\emptyset$), $\tilde g_1$ is fixed point free and
    $\Omega(g_1)=\T$, then it holds
    \begin{displaymath}
      \W(g)= \bigcup_{t\in\R} \Phi^t\big(\{0\}\times\W(g_2)\big).
    \end{displaymath}
    In particular, every connected component of the wandering set
    $\W(g)$ is lift-unbounded and inessential.
  \item\label{cond:g-fully-ess-wand-set} If $g_1$ and $g_2$ are Denjoy
    homeomorphisms, then
    \begin{displaymath}
      \W(g):= \W(g_1)\times\T \cup \bigcup_{t\in\R}
      \Phi^t\big(\{0\}\times\W(g_2)\big).  
    \end{displaymath}
    In particular, $\W(g)$ is connected and fully essential.
  \item\label{cond:Kronecker-factor-minimal-set} when the rotation
    vector of $g$ is totally irrational, the homeomorphism $g$ is a
    topological extension of a minimal translation of $\T^2$ and
    $\Omega(g)$ is the only minimal set for $g$.
  \end{enumerate}
\end{proposition}

\begin{proof}
  In order to prove \eqref{cond:g-skew-prod-struc}, it is enough to
  notice that the vertical circle foliation
  $\big\{\{t\}\times\T : t\in\R\big\}$ of $\R\times\T$ is invariant by
  the quotient map $(t,x)\mapsto \big(t+1,g_2^{-1}(x)\big)$ and the
  flow
  $\widetilde{\Phi}\colon \R\times(\R\times\T) \ni \big(t,(s,x)\big)
  \mapsto (t+s,x)$, which is the lift of flow $\Phi$.

  One can easily verify that the there exists a lift
  $\tilde g\colon\R^2\carr$ of $g$ such that
  $\rho(\tilde g)=\Big\{\big(\rho(\tilde g_1),\rho(\tilde
  g_1)\rho(\tilde g_2)\big)\Big\}$. Then this implies
  \eqref{cond:g-tot-irr-pseudo}.

  To prove \eqref{cond:g-unif-bound-rot-dev} it is enough to notice
  that $g$ leaves invariant two topologically transverse foliations
  with different asymptotic homological directions (see
  \cite[Proposition 4.1]{KocKorFoliations} and \cite[Theorem
  5.4]{KocRotDevPerPFree} for details): one of them is the vertical
  circle foliations of the skew-product structure given by
  \eqref{cond:g-skew-prod-struc}; the other one is given by the orbits
  of the (singularity free) flow $\Phi$.

  Properties \eqref{cond:g-inness-wand-set} and
  \eqref{cond:g-fully-ess-wand-set} are straightforward consequences
  of classical results about the non-wandering set of suspension
  flows.

  Finally, in order to prove
  \eqref{cond:Kronecker-factor-minimal-set}, observe that, by
  \eqref{cond:g-tot-irr-pseudo}, if $g$ is a totally irrational pseudo
  rotation, then $g_1$ and $g_2$ have irrational rotation numbers. So,
  by classical Poincaré theory, both of them are topological extension
  of irrational rotations of the circle, and $\Omega(g_1)$ and
  $\Omega(g_2)$ are the only minimal sets for $g_1$ and $g_2$,
  respectively. These properties pass to the flow $\Phi$ by
  suspension.
\end{proof}

\subsection{Unbounded inessential wandering set}
\label{sec:unbounded-iness-wandering-set}

In this paragraph we show the existence of a totally irrational
pseudo-rotation $f\in\Homeo0(\T^2)$, exhibiting uniformly bounded
rotational deviations (in every direction), whose wandering set
$\W(f)$ is connected, inessential and lift-unbounded and such that $f$
does not admit any irrational circle rotation as a topological factor.
 
To do that, let $g_2\colon\T\carr$ be a Denjoy homeomorphism with just
one orbit of wandering domains (\ie $\Per(g_2)=\emptyset$,
$\W(g_2)\neq\emptyset$ and $\W(g_2)=\bigcup_{n}g_2^n(I)$, for any
connected component $I$ of $\W(g_2)$) and
$\tilde g_1=T_{\rho_1}\colon\R\carr$ be a translation such that the
map $g\in\Homeo0(\T^2)$ constructed in \S~\ref{sec:susp-circle-homeos}
is a totally irrational pseudo-rotation (see condition
\eqref{cond:g-tot-irr-pseudo} of
Proposition~\ref{pro:g-homeo-general-properties}). Since $g_2$ has
just one orbit of wandering domains and property
\eqref{cond:g-inness-wand-set} of
Proposition~\ref{pro:g-homeo-general-properties} holds, we have that
$\W(g)$ is connected.

By \eqref{cond:Kronecker-factor-minimal-set} of
Proposition~\ref{pro:g-homeo-general-properties}, we know that $g$ is
a topological extension of a minimal rotation of $\T^2$. One can
easily show that two different points $z,z'\in\T^2$ are Kronecker
equivalent for $g$ (see Definition~\ref{def:Kronecker-equivalence}) if
and only if $\pr{1}(z)=\pr{1}(z')$ and there exists a connected
component $J$ of the set $\W(g)\cap \{\pr{1}(z)\}\times\T$ such that
$z,z'\in\overline{J}$.

So, given any topological open disc $U\subset\W(g)$, we can clearly
find two points $w_0,w_1\in U$ which are not Kronecker equivalent; and
thus, they are in fact Kronecker separated (again see
Definition~\ref{def:Kronecker-equivalence} for details).  Moreover,
taking $U$ small enough, we can assume $U$ is a wandering set for
$g$. Then, since $U$ is open and connected, there exists
$\ell\in\Homeo0(\T^2)$ so that $\supp\ell\subset U\subset\W(g)$ and
$\ell(w_0)=w_1$.

Then we just define $f:=g\circ\ell$. Since $\ell$ is supported in a
$g$-wandering set, it holds $\Omega(f)=\Omega(g)$, and therefore, $f$
and $g$ coincide on this set. So, $f$ has a connected, inessential and
lift-unbounded wandering set as well. Moreover, since $g$ exhibits
uniformly bounded rotational deviations and $\ell$ is supported on a
$g$-wandering lift-bounded set, this implies $f$ exhibits uniformly
bounded rotations deviations as well.

We claim that $f$ does not admit any irrational circle rotation as
topological factor. To prove this, reasoning by contradiction, let us
suppose there is a semi-conjugacy $h\colon\T^2\to\T$ and a minimal
rotation $T_\rho\colon\T\carr$ such that $h\circ f=T_\rho\circ
h$. Since $\Omega(f)$ is a minimal set for $f$, we have that two
points of $\Omega(f)=\Omega(g)$ are Kronecker equivalent for $f$ if
and only if so they are for $g$. Then, for each $i\in\{0,1\}$, we can
consider an arbitrary point
\begin{displaymath}
  w_i'\in\Omega(g)\cap
  \overline{\cc\left(\W(g)\cap\big\{\pr{1}(w_i)\big\}
      \times\T,w_i\right)}.   
\end{displaymath}
Since $w_0$ and $w_1$ are Kronecker separated for $g$, and $w_i$ and
$w_i'$ are $g$-proximal, for $i\in\{0,1\}$, we conclude $w_0'$ and
$w_1'$ are Kronecker separated for $g$. So, $w_0'$ and $w_1'$ are
Kronecker separated for $f$, too; but $w_0$ and $w_1'$ are
$f$-proximal and $w_0$ and $w_0'$ are $f^{-1}$-proximal. This clearly
contradicts the existence of an irrational circle factor.

\subsection{Fully essential wandering set}
\label{sec:fully-essent-wandering}

In this paragraph we describe the construction of a totally irrational
pseudo-rotation $f\in\Homeo0(\T^2)$ with uniformly bounded rotational
deviations, such that the wandering set $\W(f)$ is fully essential and
such that $f$ does not admit any irrational circle factor.

The construction is very similar to that one performed in
\S\ref{sec:unbounded-iness-wandering-set}. In this case we start
considering two Denjoy maps $g_1,g_2\in\Homeo0(\T)$, with lifts
$\tilde g_1,\tilde g_2\colon\R\carr$ such that $1$, $\rho(\tilde g_1)$
and $\rho(\tilde g_1)\rho(\tilde g_2)$ are linearly independent over
$\Q$.

Then, let $g\in\Homeo0(\T^2)$ be the homeomorphism given by the
construction described at the beginning of
\S\ref{sec:susp-circle-homeos}, associated to $\tilde g_1$ and
$g_2$. By \eqref{cond:g-tot-irr-pseudo} and
\eqref{cond:g-fully-ess-wand-set} of
Proposition~\ref{pro:g-homeo-general-properties}, we know that $g$ is
a totally irrational pseudo-rotation and $\W(g)$ is a fully essential
connected set.

Then, by \eqref{cond:g-fully-ess-wand-set} we can choose a point
$w_0\in\W(g)$ such that $\pr{1}(w_0)\in\Omega(g_1)$. So we know there
exist $I\subset\W(g_2)$, $s_0\in\R$ and $y_0\in I$ such that
\begin{displaymath}
  w_0=\Phi^{s_0}(0,y_0)\in\Phi^{s_0}\big(\{0\}\times I\big).
\end{displaymath}
Then, for any $\eps>0$ sufficiently small, the open set
\begin{displaymath}
  U:=\bigcup_{\abs{s-s_0}<\eps}\Phi^{s}\big(\{0\}\times I\big) 
\end{displaymath}
is a wandering set for $g$.

Since $\Omega(g_1)$ has no isolated points and $U$ is an open
topological disc, we know there exists a real number $s_1$ with
$0<\abs{s_1-s_0}<\eps$ such that the point
$w_1:=\Phi^{s_1}(0,y_0)\in U$, $\pr{1}(w_1)\in\Omega(g_1)$ and there
are infinitely many points in the arc flow of $\Phi$ between the
points $w_0$ and $w_1$ such that their projections on the firs
coordinate belongs to $\Omega(g_1)$, \ie the set
\begin{equation}
  \label{eq:infty-many-points-Omega-g2-arc-flow}
  \left\{t\in (0,1) :
    \pr{1}\Big(\Phi^{ts_0+(1-t)s_1}(0,y_0)\Big)\in\Omega(g_1)\right\}
\end{equation}
is not empty, (and in fact, $s_1$ can be chosen such that this set is
infinite).

Now observe that given any point $y\in\Omega(g_2)\subset\T$ and any
pair of real numbers $s<s'$, we have that $\Phi^{s}(0,y)$ and
$\Phi^{s'}(0,y)$ are Kronecker separated if an only if there exists
$t\in (s,s')$ such that $\pr{1}\big(\Phi^t(0,y)\big)\in\Omega(g_1)$.

So, if $y'\in\Omega(g_2)$ be any end point of the interval
$I\subset\W(g_2)$ we considered above, we know that the points
$\Phi^{s_i}(0,y')$ and $w_i=\Phi^{s_i}(0,y_0)$ are $g$-proximal, for
$i=0,1$. Due to our previous remark and property
\eqref{eq:infty-many-points-Omega-g2-arc-flow}, we know that the
points $\Phi^{s_0}(0,y')$ and $\Phi^{s_1}(0,y')$ are Kronecker
separated, and consequently, the points $w_0$ and $w_1$ are Kronecker
separated as well.

As we did in \S\ref{sec:unbounded-iness-wandering-set}, we consider a
homeomorphism $\ell\in\Homeo0(\T^2)$ such that $\ell(w_0)=w_1$ and
$\supp\ell\subset U\subset\W(g)$. Again, we define $f:=g\circ\ell$. By
the very same argument we exposed in
\S\ref{sec:unbounded-iness-wandering-set}, one can show $f$ is a
totally irrational pseudo-rotation exhibiting uniformly bounded
rotational deviations, fully essential wandering set and having no
irrational circle factor.

\subsection{Inessential bounded non-small wandering domains}
\label{sec:large-wandering-domains}

In this paragraph we describe the construction of a totally irrational
pseudo-rotation $f\colon\T^2\carr$ not admitting any irrational circle
factor, and such that it exhibits uniformly bounded rotational
deviations, its wandering set is the union of countably many
inessential lift-bounded wandering domains, all of them having the
same diameter and hence, not satisfying the small wandering domain
hypothesis (see Definition~\ref{def:small-wandering-domains}).

To do that, let us start by considering a totally irrational vector
$\boldsymbol{\alpha}\in\R^2$ and let
$T_{\boldsymbol{\alpha}}\colon\T^2\carr$ be the corresponding rigid
rotation. Given any $\gamma\in\R\setminus\Q$ and $\delta>0$, let us
define
\begin{displaymath}
  \mathscr{F}^\gamma_\delta:=\pi\left\{(t,t\gamma)\in\R^2 : t\in
    (-\delta,\delta)\right\}\subset \T^2.
\end{displaymath}

Notice that fixing the totally irrational vector
$\boldsymbol{\alpha}\in\R^2$, there exists $\delta_0>0$ such that
\begin{displaymath}
  T_{\boldsymbol{\alpha}}^n\big(\mathscr{F}_\delta^\gamma\big)\cap
  \mathscr{F}_\delta^\gamma =\emptyset, \quad\forall\delta\in
  (0,\delta_0),\ \forall n\in\Z\setminus\{0\}.
\end{displaymath}
Then let us fix such a $\delta$.

By classical \emph{``à la Denjoy''} surgery procedures, we can
construct a topological extension $g\in\Homeo0(\T^2)$ of
$T_{\boldsymbol{\alpha}}$ satisfying the following properties: there
exist a continuous map $h\colon\T^2\carr$ in the identity homotopy
class and a sequence $\{\alpha_n\}_{n\in\Z}$ of points of $\T^2$ such
that $h\circ g=T_{\boldsymbol{\alpha}}\circ h$, each fiber $h^{-1}(z)$
is a singleton if and only if
$z\in\T^2\setminus\bigcup_{n\in\Z}
T_{\boldsymbol{\alpha}}^n(\mathscr{F}_\delta^\gamma)$, and
\begin{displaymath}
  h^{-1}(z) = \left\{z + \alpha_n + \pi(-\gamma t, t) : t\in
    \big(-\delta_n(z),\delta_n(z)\big)\right\}, 
\end{displaymath} 
whenever $z\in T_{\boldsymbol{\alpha}}^n(\mathscr{F}_\delta^\gamma)$,
and where
$\delta_n(z):= 2^{-\abs{n}-10}\Big(\delta-d_{\T^2}\big(z,
T_{\boldsymbol{\alpha}}^n(0)\big)\Big)$ and $d_{\T^2}(\cdot,\cdot)$
denotes the distance function given by \eqref{eq:d-Td-dist-def}.

Notice that the wandering set of $g$ is given by
\begin{displaymath}
  \W(g)= \bigsqcup_{n\in\Z}
  h^{-1}\big(T_{\boldsymbol{\alpha}}^n(\mathscr{F}_\delta^\gamma)\big),
\end{displaymath}
where $\bigsqcup$ denotes the disjoint union operator, and
\begin{displaymath}
  \diam\Big(h^{-1}\big(T_{\boldsymbol{\alpha}}^n(\mathscr{F}_\delta^\gamma)
  \big)\Big) = 2\delta, \quad\forall n\in\Z.
\end{displaymath}
So, $g$ does not satisfies the small wandering domain hypothesis given
by Definition~\ref{def:small-wandering-domains}.

Then observe that two different points of $w_0,w_1\in\T^2$ are
Kronecker equivalent for $g$ if and only if there exist $n\in\Z$,
$z\in T_{\boldsymbol{\alpha}}^n(\mathscr{F}_\delta^\gamma)$ and
$t_0,t_1\in[-\delta_n(z),\delta_n(z)]$ satisfying
\begin{displaymath}
  w_i= z+\alpha_n + \pi(-\gamma t_i,t_i), \quad\text{for }
  i\in\{0,1\}.  
\end{displaymath}

In particular, inside the wandering domain
$h^{-1}(\mathscr{F}_\delta^\gamma)$ we can find two points $w_0$ and
$w_1$ which are not Kronecker equivalent. Moreover, since $\gamma$ is
irrational, this implies $w_0$ and $w_1$ are indeed Kronecker
separated.

Then, as we did in \S\S\ref{sec:unbounded-iness-wandering-set} and
\ref{sec:fully-essent-wandering}, and observing
$h^{-1}(\mathscr{F}_\delta^\gamma)$ is an open topological disc, we
can find a homeomorphism $\ell\in\Homeo0(\T^2)$ such that $h(w_0)=w_1$
and $\supp\ell\subset h^{-1}(\mathscr{F}_\delta^\gamma)$. Then we
define $f:=g\circ\ell$.

Again, $g$ clearly exhibits uniformly bounded rotational deviations
and $\ell$ is supported on a lift-bounded wandering domain, so $f$
exhibits uniformly bounded rotational deviations, too. Regarding the
wandering set, it clearly holds $\W(f)=\W(g)$. So, $f$ does not
satisfy the small wandering condition either. There exist points
$w_i'\in\Omega(g)$ such that $w_i$ is Kronecker equivalent to $w_i'$,
for $i\in\{0,1\}$ and $w_0$ is $f$-proximal to $w_1'$ and
$f^{-1}$-proximal to $w_0'$, which are Kronecker separated points. So,
$f$ does not admit any irrational circle factor.

\section{Proofs of Corollaries~\ref{cor:Jaeger-generalization} and
  \ref{cor:Dehn-twist-isotopy-class}}
\label{sec:proofs-corollaries}

In this section we prove both corollaries assuming Theorem A.

\begin{proof}[Proof of Corollary~\ref{cor:Jaeger-generalization}]
  Since $f$ is a totally irrational pseudo-rotation, we know it is
  non-eventually annular. On the other hand, since $f$ exhibits
  uniformly bounded rotational deviations, \ie condition
  \eqref{eq:pseudo-rot-unif-bound-dev} holds, we know that $f$
  satisfies the hypotheses of Theorem A, for every $v\in\bb{S}^1$ of
  rational slope.

  Then, we apply Theorem A twice in order to construct a
  ``horizontal'' and a ``vertical'' irrational circle factor, \ie we
  first consider $v=(1,0)$ and then $v=(0,1)$ to get two continuous
  semi-conjugacies $h_i\colon\T^2\to\T$, with $i=1,2$, such that
  \begin{displaymath}
    h_i\circ f = T_{\rho_i}\circ h_i, \quad\text{for } i=1,2,
  \end{displaymath}
  where $\rho(\tilde f)=(\rho_1,\rho_2)\in\R^2$ is the rotation vector
  of $\tilde f$, and $T_{\rho_i}\colon\T\carr$ is the corresponding
  irrational circle rotation.

  Finally we just define $H\colon\T^2\carr$ by
  \begin{displaymath}
    H(x,y):=\big(h_1(x),h_2(y)\big), \quad\forall (x,y)\in\T^2, 
  \end{displaymath}
  and one can easily check that $H\circ f=T_{\rho(\tilde f)}\circ H$,
  as desired.
\end{proof}

\begin{proof}[Proof of Corollary~\ref{cor:Dehn-twist-isotopy-class}]
  If $f$ is a topological extension of an irrational circle rotation,
  then by Lemma~\ref{lem:non-triv-Kron-factor-implies-BRD} we know
  that given any lift $\tilde f\colon\R^2\carr$ of $f$, there exit
  $C>0$, $v\in\bb{S}^1$ and $\rho\in\R\setminus\Q$ such that the
  estimate of that lemma holds. Since $f$ is homotopic to a Dehn twist
  $I_k$ with $k\neq 0$, we conclude that $v$ is equal to $\pm(0,1)$,
  and so, estimate \eqref{eq:bounded-vertical-rot-dev} holds.

  Reciprocally, if \eqref{eq:bounded-vertical-rot-dev} holds and
  taking into account $f$ is homotopic to a Dehn twist, we conclude
  $f$ is not eventually annular and hence we can directly apply
  Theorem A to conclude the existence of an irrational circle factor.
\end{proof}

\section{$\rho$-centralized skew-product}
\label{sec:rho-centralized-skew}

In this section we introduce the main character of this work: the
$\rho$-centralized skew-product induced by the torus homeomorphism
$f$. This is a generalization of that one we introduced in
\cite{KocMinimalHomeosNotPseudo,KocRotDevPerPFree} to study rotational
deviations for periodic point free homeomorphisms in the identity
homotopy class.

The intuitive idea behind the definition of this skew-product is to
separate the rotational component of a given torus homeomorphism from
its complement which is responsible for the sub-linear rotational
deviations. This idea can be easily formalized regarding the group of
area-preserving $2$-torus homeomorphisms which are homotopic to the
identity. In fact, this group can be written as the semi-direct
product of the group of rotations (which is isomorphic to $\T^2$
itself) and the group of Hamiltonian homeomorphisms, \ie the group of
area-preserving homeomorphisms in the identity homotopy class with
vanishing Lebesgue rotation vector (also called \emph{flux} in
symplectic geometry); and this group-theoretically decomposition can
be taken to the dynamical system setting defining a skew-product
homeomorphism on $\T^2\times\R^2$ which is just a rotation on the base
and acts by (lifts of) Hamiltonian homeomorphisms on the fibers.

The main heuristic of this skew-product can be summarized as follows:
bounded orbits of this skew-product correspond to orbits of the
original homeomorphism with bounded rotational deviations with respect
to the rotation vector that acts on the base of the skew-product. The
reader can find in \cite[\S 6.1]{KocMinimalHomeosNotPseudo} all the
details of this construction, especially those relating the
semi-direct group decomposition of the group of area-preserving
$2$-torus homeomorphism and the definition of the induced
skew-product, that we called \emph{fiberwise Hamiltonian skew-product}
there. Later, in \cite{KocRotDevPerPFree} we extended this
construction for any homeomorphism in the identity isotopy class and
any $\rho\in\R^2$, and we started calling them
\emph{$\rho$-centralized skew-product,} which is the name we shall
continue using from now on.

In this work we will introduce a slight modification of the
construction we performed in \cite{KocRotDevPerPFree} in order to deal
with homeomorphisms in Dehn twist isotopy classes. In fact, as we saw
in \S\ref{sec:rot-set-rot-vect}, when the homeomorphism
$f\in\Homeo{k}(\T^2)$ with $k\neq 0$, one cannot define rotation
vectors as elements of $\R^2$, but just vertical rotation numbers as
given by \eqref{eq:rot-vertical-def}. So, here 
$\rho$-centralized skew-products will be defined on $\T^2\times\A$ and
not on $\T^2\times\R^2$ as done in \cite{KocMinimalHomeosNotPseudo,
  KocRotDevPerPFree}.

Summarizing the results we are going to find in this section, let
$f\in\Homeo{k}(\T^2)$ be a given homeomorphism and
$F\colon\T^2\times\A\carr$ denote the induced $\rho$-centralized
skew-product (that we will formally introduce just after this brief
summary). In Proposition~\ref{pro:F-Hs-commutes} we study the group of
symmetries of $F$, showing it includes a full flow called
$\Gamma$. Then, in Proposition~\ref{pro:F-orbit-bound-vert-dir} we
prove that there is a close relation between points of $\T^2$
exhibiting bounded vertical rotational deviations for $f$ and points
of $\T^2\times\A$ whose $F$-orbits are bounded. One of the main
results of this section is Theorem~\ref{thm:Omega-f-vs-Omega-F} where
we show that under the $\Omega$-recurrence hypothesis, there is
natural correspondence between non-wandering points $f$ and
non-wandering points of $F$, showing in particular that $f$ is a
non-wandering homeomorphism if and only if $F$ is. Finally, in
Theorem~\ref{thm:separating-sets-orbits-of-open-sets} we study the
topology of $F$-invariant bounded connected sets, showing that under
the hypotheses of Theorem A, the complement of theses sets have
exactly two unbounded connected components. The boundary of these sets
will be used in \S\ref{sec:proof-thm-A} to construct the fibers of the
semi-conjugacy with the irrational circle factor.


Let us state now the precise definition of the induced
\emph{$\rho$-centralized skew-product}. Let $k$ be any integer number,
$f\in\Homeo{k}(\T^2)$ an arbitrary homeomorphism and
$\tilde f\colon\R^2\carr$ be a lift of $f$. We know that the
displacement function $\Delta_{\tilde f}:=\tilde f-I_k\colon\R^2\carr$
is $\Z^2$-periodic, and thus, it can be considered as an element of
$C^0(\T^2,\R^2)$, where $I_k\in\SL(2,\Z)$ is given by
\eqref{eq:Im-def}.

For the sake of simplicity, let us write
\begin{displaymath}
  \Delta_i:=\pr{i}\circ\Delta_{\tilde f}, \quad\text{for } i=1,2.
\end{displaymath}

On the other hand, let $\hat f\colon\A\carr$ be the only annulus
homeomorphism such that $\tilde\pi\circ\tilde f=\hat f\circ\tilde\pi$,
where $\tilde\pi$ is the covering map given by
\eqref{eq:tilde-pi-def}.

Then, given any $\rho\in\R$ let us define the \emph{$\rho$-centralized
  skew-product} induced by $\tilde f$ as the homeomorphism
$F\colon\T\times\A\carr$ given by
\begin{equation}
  \label{eq:rho-centr-skew-prod-def}
  F(t,x,\tilde y):=\Big(T_\rho(t), x + k(y+t) +
  \pi\Big(\Delta_1\big(x,y+t\big)\Big), \tilde y +
  \Delta_2(x,y+t)-\rho \Big),
\end{equation}
for every $(t,x,\tilde y)\in\T\times\A = \T\times\T\times\R$ and where
$y:=\pi(\tilde y)\in\T$.

The usefulness of map $F$ can be briefly summarized with the following
simple but important property:

\begin{equation}
  \label{eq:F-fundamental-prop}
  F^n(t,\hat z) = \Big(T_\rho^n(t), \hat T_{\tilde t}^{-1}\circ
  \big(\hat T_\rho^{-n}\circ\hat f^n\big) \circ\hat T_{\tilde
    t}(\hat z)\Big), 
\end{equation}
for any $(t,\hat z)\in\T\times\A$, any $\tilde t\in\pi^{-1}(t)$ and
every $n\in\Z$.

On the other hand, let us notice that the skew-product $F$ has a full
flow of symmetries: for each $s\in\R$ consider the map
$\Gamma^s\colon\T\times\A\carr$ given by
$\Gamma^s:=T_s\times\hat T_{-s}$, \ie
\begin{equation}
  \label{eq:Hs-definition}
  \Gamma^s(t,x,\tilde y):=\big(t+\pi(s), x, \tilde y-s) = \big(T_s(t),
  \hat{T}_{-s}(x,\tilde y)\big), \quad\forall
  (t,x,\tilde y)\in\T\times\A.
\end{equation}
Then, we have the following

\begin{proposition}
  \label{pro:F-Hs-commutes}
  The map $\Gamma\colon\R\times\T\times\A\to \T\times\A$ is a flow
  and, for each $s\in\R$, the homeomorphisms $\Gamma^s$ and $F$
  commute, and $\Gamma^s$ is an isometry of the metric space
  $(\T\times\A, d_{\T\times\A})$, where the distance function is given
  by
  \begin{displaymath}
    d_{\T\times\A}\big((t,z),(t',z')\big) = d_\T(t,t')+d_\A(z,z'),
    \quad\forall (t,z),(t',z')\in\T\times\A.  
  \end{displaymath}
\end{proposition}

\begin{proof}
  This follows from straightforward computations. In fact,
  \begin{displaymath}
    \Gamma^{s+s'}= T_{s+s'}\times \hat{T}_{-s-s'} =
    \big(T_s\times\hat{T}_{-s}\big)\circ
    \big(T_{s'}\times\hat{T}_{-s'}\big) =
    \Gamma^s\circ\Gamma^{s'}, \quad\forall s,s'\in\R.
  \end{displaymath}

  So, given any $s\in\R$ and any
  $(t,x,\tilde y)\in\T\times\A$, it holds
  \begin{displaymath}
    \begin{split}
      &F\big(\Gamma^s(t,x,\tilde y)\big) = F\big(t+\pi(s),x,\tilde y
      -s\big) \\
      &=\Big(T_{\rho}\big(T_s(t)\big), x + k(y+t) +
      \pi\circ\Delta_1(x,y+t), \tilde y -s
      +\Delta_2(x,y+t)-\rho\Big)  \\
      & = \Gamma^s\big(F(t,x,\tilde y)\big).
    \end{split}
  \end{displaymath}

  Finally, $\Gamma^s$ is an isometry because $T_s$ is an isometry of
  $(\T,d_\T)$ and $\hat T_ s$ of $(\A,d_\A)$.
\end{proof}

Then, let us fix some terminology. Given any
$(t,\hat z)\in\T\times\A$, the \emph{$\Gamma$-line through}
$(t,\hat z)$ is its flow line, \ie it is given by
\begin{displaymath}
  \Gamma(t,\hat z):=\left\{\Gamma^s(t,\hat z) \in\T\times\A :
    s\in\R\right\}. 
\end{displaymath}
On the other hand, we introduce the concept of \emph{blocks} of
$\T\times\A$, which are a particular kind of open subsets of
$\T\times\A$: given an open set $V\subset\A$, a point $t\in\T$ and a
positive real number $r$, we define the corresponding \emph{$r$-block
  centered at $t$} associated to $V$ by
\begin{equation}
  \label{eq:V-r-block}
  V^{r,t}:=\bigcup_{\abs{s}<r} \Gamma^s\big(\{t\}\times
  V\big)\subset\T\times\A. 
\end{equation}

A rather simple but important property of blocks is given by the
following
\begin{proposition}
  \label{pro:F-image-of-r-block}
  If $V\subset\A$ is an open subset, $t\in\T$ and $r>0$, and $V^{r,t}$
  is the $r$-block centered at $t$ associated to $V$ given by
  \eqref{eq:V-r-block}, then it holds
  \begin{displaymath}
    F^n\big(V^{r,t}\big)=\big(\hat T_{\tilde t}^{-1}\circ\hat
    T_{\rho}^{-n}\circ \hat f^n\circ\hat T_{\tilde
      t}(V)\big)^{r,T_\rho^n(t)}, \quad\forall n\in\Z,\ \forall\tilde
    t\in\pi^{-1}(t).
  \end{displaymath}
  In particular, the $F$-image of any $r$-block is another $r$-block.
\end{proposition}

\begin{proof}
  It easily follows from \eqref{eq:F-fundamental-prop},
  Proposition~\ref{pro:F-Hs-commutes} and \eqref{eq:V-r-block}.
\end{proof}

Finally, as a straightforward consequence of
\eqref{eq:F-fundamental-prop} we get the following
\begin{proposition}
  \label{pro:f-vs-F-dynamics}
  Given any $z\in\T^2$, any $n\in\Z$, any $\hat z\in\hat\pi^{-1}(z)$
  and any $\hat w\in\hat\pi^{-1}\big(f^n(z)\big)$ (where projection
  $\hat\pi$ is given by \eqref{eq:hat-pi-def}), the point
  $F^n(0,\hat z)$ belongs to the $\Gamma$-line through $(0,\hat w)$.
\end{proposition}

\begin{proof}
  Since $\hat\pi\circ\hat f=f\circ\hat\pi$, we have
  \begin{displaymath}
    \hat\pi\big(\hat f^n(\hat z)\big) = f^n\big(\hat\pi(\hat z)\big) =
    f^n(z) = \hat\pi(\hat w).
  \end{displaymath}
  So there exists $m\in\Z$ such that
  $\hat T_m(\hat w) = \hat f^n(\hat z)$, or equivalently we can write
  \begin{displaymath}
    \Gamma^{-m}(0,\hat w)= \big(0,\hat f^n(\hat z)\big). 
  \end{displaymath}
  
  Then, by \eqref{eq:F-fundamental-prop} we know that
  \begin{displaymath}
    F^n(0,\hat z)= \big(n\pi(\rho), \hat T_\rho^{-n}\circ\hat f^n(\hat
    z)\big) = \Gamma^{n\rho}\big(0,\hat f^n(\hat z)\big) =
    \Gamma^{n\rho-m}(0,\hat w). 
  \end{displaymath}
\end{proof}

\subsection{The $\rho$-centralized skew-product and rotational
  deviations}
\label{sec:rho-centr-skew-prod-rot-dev}

From now on and until the end of this section we suppose
$\tilde f\colon\R^2\carr$ is a lift of a homeomorphism
$f\in\Homeo{k}(\T^2)$ exhibiting uniformly bounded vertical
deviations, with $\rho$, $C$ and $v=(0,1)$ as in
\eqref{eq:unif-bound-v-dev-def}. Let
$F\colon\T\times\A\carr$ be the $\rho$-centralized skew-product
induced by $\tilde f$ as defined in
\eqref{eq:rho-centr-skew-prod-def}.

Then we have the following
\begin{proposition}
  \label{pro:F-orbit-bound-vert-dir}
  Every $F$-orbit is bounded in the vertical direction, \ie if
  $(t,x,\tilde y)\in\T\times A$ is an arbitrary point and we define
  $(t_n,x_n,\tilde y_n):=F^n(t,x,\tilde y)$, then it holds
  \begin{displaymath}
    \abs{\tilde y_m-\tilde y_n}\leq 2C, \quad\forall m,n\in\Z,
  \end{displaymath}
  where $C$ is the constant given by \eqref{eq:unif-bound-v-dev-def}.
\end{proposition}

\begin{proof}
  This is a straightforward consequence of estimate
  \eqref{eq:unif-bound-v-dev-def} and property
  \eqref{eq:F-fundamental-prop}.
\end{proof}

The following result will play a key role in our work:

\begin{theorem}
  \label{thm:Omega-f-vs-Omega-F}
  If $f$ is $\Omega$-recurrent (see
  Definition~\ref{def:non-wandering-recurrent}), then a point
  $(x,y)\in\T^2$ is non-wandering for $f$ if and only if
  $\Gamma^s(0,x,\tilde y)\in\Omega(F)$, for every
  $\tilde y\in\pi^{-1}(y)$ and all $s\in\R$.
\end{theorem}

\begin{proof}
  First observe that, by Proposition~\ref{pro:F-Hs-commutes}, the set
  $\Omega(F)$ is $\Gamma$-invariant. So it is enough to show that a
  point $(x,y)\in\Omega(f)$ if and only if
  $(0,x,\tilde y)\in\Omega(F)$, for some $\tilde y\in\pi^{-1}(y)$.
  
  Then, let us prove the ``if'' direction, which holds without the
  boundedness of rotational deviations and the $\Omega$-recurrence
  assumptions. So, let $(0,x,\tilde y)$ be any point of $\Omega(F)$.

  Let us fix a positive number $\delta$ and write just $B$ for the
  open ball $B_\delta(x,\tilde y)\subset\A$. Then, since
  $(0,x,\tilde y)$ is a non-wandering point, there exists $n\geq 1$
  such that
  \begin{displaymath}
    F^n\big(B^{\delta,0}\big)\cap B^{\delta,0}\neq\emptyset,
  \end{displaymath}
  where $B^{\delta,0}$ denotes the $\delta$-block centered at $0\in\T$
  and associated to $B$, as defined by \eqref{eq:V-r-block}. So, there
  are $s,s'\in (-\delta,\delta)$ such that
  $T_\rho^n\big(\pi(s)\big)=\pi(s')$ and the sets $F^n(B^{\delta,0})$
  and $B^{\delta,0}$ intersect on the fiber $\{s'\}\times\A$. Thus,
  by \eqref{eq:F-fundamental-prop}, we have
  \begin{displaymath}
    \hat{T}_{s}^{-1}\circ\hat{T}_\rho^{-n}\circ\hat{f}^n\circ\hat{T}_s
    \big(\hat{T}_{-s}(B)\big)\cap\hat{T}_{-s'}(B)\neq\emptyset, 
  \end{displaymath}
  and hence,
  \begin{displaymath}
    \hat{f}^n(B)\cap\hat{T}_{n\rho+s-s'}(B) \neq\emptyset.  
  \end{displaymath}
  
  However, since $T_\rho^n\big(\pi(s)\big)=\pi(s')$, we know there
  exists $p\in\Z$ such that $n\rho+s=s'+p$, and then, recalling $B$
  denotes the ball of radius $\delta$ and center $(x,\tilde y)$ in
  $\A$, we get
  \begin{displaymath}
    f^n\big(B_{\delta}(x,y)\big)\cap B_{\delta}(x,y)\neq\emptyset,
  \end{displaymath}
  where $y=\pi(\tilde y)$. Therefore, $(x,y)\in\Omega(f)$, as desired.
  
  Now, let $(x,y)$ be any point of $\Omega(f)$ and let us fix a point
  $\tilde y\in\pi^{-1}(y)$. Given any real number $r>0$, let $B_r$
  denote the open ball $B_r(x,\tilde y)\subset\A$ and $B_r^{r,0}$ be
  the $r$-block centered at $0$ associated to $B_r$. Observe the
  family of blocks $\big\{ B_r^{r,0} : r>0\big\}$ is a local base of
  neighborhoods at the point $(0,x,\tilde y)$. So, let us fix a real
  number $r>0$ that, without loss of generality, we can suppose is
  less than $1/4$, and let us show there exists $n\in\N$ such that
  $B_r^{r,0}\cap F^{-n}(B_r^{r,0})\neq\emptyset$.

  Since $f$ is $\Omega$-recurrent, by
  Lemma~\ref{lem:non-wandering-vs-recurrent-pts} there exists an
  $f$-recurrent point $(x',y')\in B_r(x,y)$ and a strictly increasing
  sequence of positive integers $(n_j)_{j\geq 1}$ such that
  \begin{equation}
    \label{eq:xj-yj-recurrent-points}
    (x_j',y_j'):=f^{n_j}(x',y')\in B_r(x,y), \quad\forall j\in\N.
  \end{equation}

  Since $r<1/4$, there is a unique point
  $\tilde y'\in\pi^{-1}(y')\cap B_{r}(\tilde y)$, and for each
  $j\geq 1$, a unique point
  $\tilde y_j'\in\pi^{-1}(y_j')\cap B_r(\tilde y)$.

  By Proposition~\ref{pro:f-vs-F-dynamics}, for each $j\geq 1$ there
  is a real number $s_j$ such that
  \begin{displaymath}
    F^{n_j}(0,x',\tilde y')=\Gamma^{s_j}(0,x'_j,\tilde y'_j), \quad\forall
    j\geq 1. 
  \end{displaymath}
  On the other hand, by Proposition~\ref{pro:F-orbit-bound-vert-dir}
  we know that
  \begin{displaymath}
    \abs{s_j}\leq 2C,\quad\forall j\in\N. 
  \end{displaymath}
  Now, invoking Proposition~\ref{pro:F-Hs-commutes} we get
  \begin{displaymath}
    F^{-n_j}\bigg(\bigcup_{\abs{s}<r} \Gamma^s(0,x_j',\tilde y_j')\bigg)
    = \bigcup_{\abs{s}<r} \Gamma^s\big(F^{-n_j}(0,x_j',\tilde y_j')\big)    
    \subset \bigcup_{\abs{s}\leq 2C+r} \Gamma^s(0,x',\tilde y'),
  \end{displaymath}
  for every $j\geq 1$.

  Now, recalling $\Gamma$ is an isometric flow (\ie for every $s$,
  $\Gamma^s $ leaves invariant the distance $d_{\T\times\A}$ defined
  in Proposition~\ref{pro:F-Hs-commutes}), we observe the arc on the
  right side of the equation has finite length. On the other hand, on
  the left side of the equation, we have infinitely many constant
  length segments. So we can conclude there exist $k>j\geq 1$ such
  that
  \begin{displaymath}
    F^{-n_k}\bigg(\bigcup_{\abs{s}<r} \Gamma^s\big(0,x_k',\tilde
    y_k'\big)\bigg) \cap F^{-n_j}\bigg(\bigcup_{\abs{s}<r}
    \Gamma^s\big(0,x_j',\tilde y_j'\big)\bigg) \neq\emptyset. 
  \end{displaymath}
  Since $(0,x_j',\tilde y_j'), (0,x_k',\tilde y_k')\in B_r^{r,0}$,
  this implies $F^{n_k-n_j}(B_r^{r,0})\cap B_r^{r,0}\neq\emptyset$,
  with $n_k-n_j>0$. So, $(0,x,\tilde y)\in\Omega(F)$.
\end{proof}

Now we can state the main result of this section:

\begin{theorem}
  \label{thm:separating-sets-orbits-of-open-sets}
  Let us suppose $f\in\Homeo{k}(\T^2)$ is $\Omega$-recurrent, periodic
  point free, exhibits uniformly bounded vertical rotational
  deviations and is not eventually annular. Let $V\subset\T\times\A$
  be a nonempty connected bounded open set such that
  $V\cap\Omega(F)\neq\emptyset$. Then, for each $t\in\T$, the open set
  \begin{displaymath}
    \A\setminus\overline{\U_F(V)}_t
  \end{displaymath}
  has exactly two unbounded connected components, where
  $\U_f(V)\subset\T\times\A$ is the set given by \eqref{eq:U-eps-def}
  and $\U_F(V)_t$ denotes the fiber of $\U_F(V)$ over $t$ given by
  \eqref{eq:fiber-A-over-x}.

  As a consequence of this, one concludes the set $\U_F(V)$ is in fact
  $F$-invariant.
\end{theorem}

In order to prove
Theorem~\ref{thm:separating-sets-orbits-of-open-sets}, let us start
considering the following

\begin{lemma}
  \label{lem:separating-sets-orbits-of-open-sets-tot-irrat-bound-dev}
  If $f\in\Homeo0(\T^2)$ is an $\Omega$-recurrent totally irrational
  pseudo-rotation exhibiting uniformly bounded rotational deviations,
  \ie estimate \eqref{eq:pseudo-rot-unif-bound-dev} holds, and $V$ is
  as in Theorem~\ref{thm:separating-sets-orbits-of-open-sets}, then
  $\U_F(V)_t\subset\A$ is an annular set, for every $t\in\T$.
\end{lemma}

\begin{proof}[Proof of
  Lemma
  \ref{lem:separating-sets-orbits-of-open-sets-tot-irrat-bound-dev}] 
  Let $\tilde f\colon\R^2\carr$ be a lift of $f$ and $\bar\rho\in\R^2$
  be the totally irrational vector appearing in estimate
  \eqref{eq:pseudo-rot-unif-bound-dev}. So, $\rho=\pr{2}(\bar\rho)$ is
  the irrational number we are considering to define the
  $\rho$-centralized skew-product given by
  \eqref{eq:rho-centr-skew-prod-def}.

  Reasoning by contradiction, suppose there exists a nonempty open
  connected subset $V\subset\T\times\A$ and $t\in\T$ such $\U_F(V)_t$
  is not annular, \ie is inessential in $\A$. Since $\rho$ is
  irrational and $\U_F(V)$ is open, we conclude $\U_F(V)_t$ is
  inessential in $\A$, for every $t\in\T$.
  
  Consider the covering map
  $\pi\times\tilde\pi\colon\R\times\R^2\to\T\times\A$ given by
  \begin{equation}
    \label{eq:pi-R3-covering-T-A}
    \pi\times\tilde\pi(\tilde t,\tilde z):=\big(\pi(\tilde t),
    \tilde\pi(\tilde z)\big), \quad\forall (\tilde t,\tilde
    z)\in\R\times\R^2,
  \end{equation}
  and let $\widetilde{\U_f(V)}\subset\R^3$ be a connected component of
  $(\pi\times\tilde\pi)^{-1}\big(\U_F(V)\big)$. Again by irrationality
  of $\rho$, it must hold $\pr{1}\big(\U_F(V)\big)=\T$. On the other
  hand, since $\pi\times\tilde\pi$ is a covering map, and both sets
  $\U_F(V)$ and $\widetilde{\U_F(V)}$ are open and connected, we get
  \begin{equation}
    \label{eq:pi-image-tilde-U-is-U}
    \pi\times\tilde\pi\Big(\widetilde{\U_F(V)}\Big) = \U_F(V). 
  \end{equation}
  So we have $\pr{1}\Big(\widetilde{\U_F(V)}\Big)=\R$.

  On the other hand, since $\U_F(V)$ is open, connected and bounded in
  $\T\times\A$, we know that $\pr{3}\Big(\widetilde{\U_F(V)}\Big)$ is
  also bounded in $\R$. Hence, there exists $\ell\in\Z$, which is
  unique and independent of the choice of the connected component
  $\widetilde{\U_F(V)}$, such that
  \begin{equation}
    \label{eq:ell-holonom-UFV-tilde}
    \widetilde{\U_F(V)}_{\tilde t+1} =
    T_{(\ell,0)}\Big(\widetilde{\U_F(V)}_{\tilde t}\Big), \quad\forall
    \tilde t\in\R.
  \end{equation}

  Since $V$ intersects $\Omega(F)$ and
  \eqref{eq:pi-image-tilde-U-is-U} holds, there is
  $(\tilde t,\tilde z)\in\widetilde{\U_F(V)}$ so that the point
  $(t,z):=\pi\times\tilde\pi(\tilde t,\tilde z)$ belongs to
  $\Omega(F)$. This implies there exists a sequence of positive
  integers $(n_j)_{j\geq 1}$ and a sequence of points
  $\big((t_j,z_j)\big)_{j\geq 1}$ such that
  $(t_j,z_j)\in B_{1/4j}(t,z)\subset\T\times\A$ and
  $F^{n_j}(t_j,z_j)\in B_{1/4j}(t,z)$, for each $j\geq 1$. Observe,
  since $F$ is periodic point free, it necessarily holds
  $n_j\to+\infty$, as $j\to\infty$. So, by
  \eqref{eq:F-fundamental-prop}, we can conclude that for each
  $j\geq 1$ there exist unique numbers $m_j,p_j\in\Z$ such that
  $\abs{m_j-n_j\rho}<1/2j$ and
  \begin{equation}
    \label{eq:f-nj-tj-z-j-displacement-estimate}
    \frac{1}{n_j}\norm{\tilde f^{n_j}(\tilde z_j+(0,\tilde t_j))-
      \tilde z_j - (0,\tilde t_j) - n_j(0,\rho) - (p_j,0)}\leq
    \frac{1}{2jn_j},  
  \end{equation}
  where $(\tilde t_j,\tilde z_j)$ is the only point in
  $(\pi\times\tilde\pi)^{-1}(t_j,z_j)\cap B_{1/j}(\tilde t,\tilde z)$,
  provided $j\geq 4$.
  
  Now, recalling we are assuming $f$ is a pseudo-rotation and
  $\rho(\tilde f)=\{\bar\rho\}$ is the rotation set given by
  \eqref{eq:rho-set-def}, as a consequence of
  \eqref{eq:f-nj-tj-z-j-displacement-estimate} we get
  $p_j/n_j\to\pr{1}(\bar\rho)=\rho$ and $m_j/n_j\to\pr{2}(\bar\rho)$,
  as $j\to\infty$.

  However observe that, by \eqref{eq:ell-holonom-UFV-tilde},
  $p_j=\ell m_j$, for all $j\geq 1$.  This implies that
  $\pr{1}(\bar\rho)=\ell\pr{2}(\bar\rho)$, contradicting the fact that
  $\overline{\rho}$ was totally irrational. So, $\U_F(V)_t$ is annular
  for every $t\in\T$.
\end{proof}

Then we can finish the proof of
Theorem~\ref{thm:separating-sets-orbits-of-open-sets}:

\begin{proof}[Proof of Theorem~\ref{thm:separating-sets-orbits-of-open-sets}]
  If $f\in\Homeo0(\T^2)$ is a pseudo-rotation and exhibits uniformly
  bounded horizontal deviations, recalling vertical deviations are
  uniformly bounded as well, then it exhibits uniformly bounded
  rotational deviations in every direction. Since $f$ is not
  eventually annular, this implies the rotation set of any lift of $f$
  is singleton containing a totally irrational vector. Hence we can
  invoke
  Lemma~\ref{lem:separating-sets-orbits-of-open-sets-tot-irrat-bound-dev}
  to guarantee that $\U_F(V)_t$ is annular for every $t\in\T$, and
  thus, $\A\setminus\overline{\U_F(V)}_t$ has exactly two unbounded
  connected components; and the theorem is proven in this case.
  
  So, now we can assume that $f$ exhibits unbounded horizontal
  rotational deviations, \ie if $\tilde f\colon\R^2\carr$ is a lift of
  $f$, then condition \eqref{eq:horizontal-unbounded-deviations}
  holds.

  Reasoning by contradiction, let us suppose there exists a nonempty
  open connected bounded subset $V\subset\T\times\A$ and $t_0\in\T$,
  such that $V\cap\Omega(F)\neq\emptyset$ and
  $\A\setminus\overline{\U_F(V)}_{t_0}$ has just one unbounded
  connected component.

  Since $f$ exhibits uniformly bounded vertical deviations, by
  \eqref{eq:F-fundamental-prop} we know the set $\U_F(V)$ is bounded
  in $\T\times\A$ and so there exists a constant $C_v>0$ such that
  \begin{equation}
    \label{eq:U-F-vertical-boundedness}
    \U_F(V)\subset \T\times \big(\T\times (-C_v,C_v)\big). 
  \end{equation}

  Thus, since $\A\setminus\overline{\U_F(V)}_{t_0}$ is open, its
  unique unbounded connected component is arc-wise connected as
  well. So there exists a continuous curve $\gamma\colon [0,1]\to\A$
  such that $\gamma(0)\in \T\times (-\infty, -C_v-1)$,
  $\gamma(1)\in\T\times (C_v+1,+\infty)$ and
  $\gamma(s)\not\in\overline{\U_F(V)}_{t_0}$, for every $s\in
  [0,1]$. By compactness of the (image of the) curve $\gamma$ and the
  set $\overline{\U_F(V)}$, there exists $\eps>0$ such that
  \begin{displaymath}
    \gamma(s)\not\in\overline{\U_F(V)}_t, \quad\forall t\in
    B_\eps(t_0),\ \forall s\in [0,1], 
  \end{displaymath}
  where $B_\eps(t_0):=\{t\in\T : d_\T(t,t_0)<\eps\}$. Since the number
  $\rho$ appearing in the base dynamics of the skew-product $F$ is
  irrational, this implies that there exists $N\in\N$ such that
  \begin{displaymath}
    \bigcup_{j=0}^N T_\rho^j\big(B_\eps(t_0)\big) = \T. 
  \end{displaymath}
  So, for every $t\in\T$ there are $n_t\in\{0,1,\ldots, N\}$ and
  $t'\in B_\eps(t_0)$ such that $T_\rho^{n_t}(t')=t$ and then,
  \begin{equation}
    \label{eq:gamma-interate-separetate-vertically-U-F}
    F^{n_t}\big(t',\gamma(s)\big)\not\in\overline{\U_F(V)},
    \quad\forall s\in [0,1]. 
  \end{equation}

  Then, putting together \eqref{eq:U-F-vertical-boundedness} and
  \eqref{eq:gamma-interate-separetate-vertically-U-F} we conclude that
  for every $t\in\T$, every connected component of the set $\U_F(V)_t$
  is lift-bounded in $\A$. Moreover, they are uniformly bounded, \ie
  there is a real constant $C>0$ such that $\diam \widetilde{U}<C$,
  for every $t\in\T$ and every connected component $\widetilde U$ of
  the open set $\tilde\pi^{-1}\big(\U_F(V)_t\big)\subset\R^2$, where
  $\tilde\pi\colon\R^2\to\A$ denotes the covering map given by
  \eqref{eq:hat-pi-def}.

  Let us see this leads us to a contradiction. In fact, by
  Proposition~\ref{pro:F-Hs-commutes} we know the set $\Omega(F)$ is
  $\Gamma$-invariant, and we are assuming
  $V\cap\Omega(F)\neq\emptyset$. So, there exists $s\in\R$ such that
  $\Gamma^s(V)\cap\Omega(F)\cap\big(\{0\}\times\A\big)\neq\emptyset$. Let
  us consider the fiber of $\Gamma^s(V)$ over the point $0\in\T$, \ie
  the set $\big(\Gamma^s(V)\big)_0$. Since $f$ is $\Omega$-recurrent,
  by Theorem \ref{thm:Omega-f-vs-Omega-F} we know that
  \begin{displaymath}
    \big(\Gamma^s(V)\big)_0\cap\hat\pi^{-1}\big(\Omega(f)\big)
    \neq\emptyset. 
  \end{displaymath}
  Hence, there is a connected component $\tilde{V}$ of the open set
  $\tilde\pi^{-1}\Big(\big(\Gamma^s(V)\big)_0\Big)\subset\R^2$ such
  that $\pi(\tilde V)\cap\Omega(f)\neq\emptyset$, where
  $\tilde\pi\colon\R^2\to\A$ and $\pi\colon\R^2\to\T^2$ are the
  natural covering maps.
  
  Now notice, since we are assuming $f$ is periodic point free and
  non-eventually annular, by
  Proposition~\ref{pro:ppf-nonannular-homeos-fully-essential-points},
  every non-wandering point is fully essential, and recalling $f$
  exhibits unbounded horizontal rotational deviations (\ie condition
  \eqref{eq:horizontal-unbounded-deviations} holds), we can invoke
  Lemma~\ref{lem:eventually-large-non-wandering-open-sets} to conclude
  that
  \begin{equation}
    \label{eq:V0-infinite-diameter-estimate}
    \sup_{n\in\Z} \diam\Big(\pr{1}\big(\tilde
    f^n(\tilde V)\big)\Big) = \infty,
  \end{equation}

  Finally, by \eqref{eq:F-fundamental-prop} we know that
  \eqref{eq:V0-infinite-diameter-estimate} contradicts the fact that
  every connected components of $\overline{\U_F(V)}_t$ are uniformly
  lift-bounded, for each $t\in\T$. Thus, we have proved that
  $\A\setminus\overline{\U_F(V)}_t$ has exactly two unbounded
  components for every $t\in\T$.

  In order to finish the proof of
  Theorem~\ref{thm:separating-sets-orbits-of-open-sets}, it just
  remains to show that the set $\U_F(V)$ is $F$-invariant. To do this,
  by the very definition of $\U_F(V)$ (given by \eqref{eq:U-eps-def}),
  we know that $\U_F(V)$ is a periodic set and
  $F\big(\U_F(V)\big)\cap \U_F(V)\neq\emptyset$ if and only if
  $F\big(\U_F(V)\big)=\U_F(V)$.

  So, reasoning by contradiction, if we suppose $\U_F(V)$ is not
  $F$-invariant, then we get that $F\big(\U_F(V)\big)$ and $\U_F(V)$
  are disjoint. We have already proved that the complement of the
  closed set $\overline{\U_F(V)}_t$ has exactly two unbounded
  components on the annulus $\A$, for every $t\in\T$, which are called
  the upper and lower components and are denoted by
  $U^+\big(\overline{\U_F(V)}_t\big)$ and
  $U^-\big(\overline{\U_F(V)}_t\big)$, respectively. So, if $\U_F(V)$
  is not invariant by $F$, since $\overline{\U_F(V)}$ is connected
  then we have either
  $F\big(\U_F(V)_t\big)\subset
  U^+\big(\overline{\U_F(V)}_{T_\rho(t)}\big)$ or
  $F\big(\U_F(V)_t\big)\subset
  U^-\big(\overline{\U_F(V)}_{T_\rho(t)}\big)$, for every
  $t\in\T$. Without loss of generality, if we assume the first
  inclusion holds for every $t\in\T$, we conclude that
  \begin{displaymath}
    F\Big(\overline{U^+\big(\overline{\U_F(V)}_t\big)}\Big) \subset
    U^+\big(\overline{\U_F(V)}_{T_\rho(t)}\big), \quad\forall t\in\T.
  \end{displaymath}

  This implies the open set $U^+$ given by
  \begin{displaymath}
    U^+:=\bigcup_{t\in\T} \{t\}\times U^+\big(\overline{\U_F(V)}_t\big)
    \subset\T\times\A, 
  \end{displaymath}
  satisfies $F\big(\overline{U^+}\big)\subset U^+$ and this clearly
  contradicts the fact there exists $n>1$ such that
  $F^n\big(\U_F(V)\big)=\U_F(V)$.
\end{proof}

\section{Proof of Theorem A}
\label{sec:proof-thm-A}

In order to show that estimate \eqref{eq:v-bounded-deviation} is a
necessary condition for the existence of an irrational circle factor,
no assumption about the non-wandering set is required. In fact, the
proof is based in rather classical arguments and all the details can
be found for instance in \cite[Lemma 3.1]{JaegerTalIrratRotFact}. Let
us just mention that in this case the irrationality of $\rho$ follows
from the non-annularity hypothesis.

To finish the proof of Theorem A, from now on let us suppose
$\tilde f$ is a lift of $f$ and satisfies condition
\eqref{eq:v-bounded-deviation}. By
Proposition~\ref{pro:periodic-point-free-homeos-conj-Hk} and the
remark about rotational deviations at the beginning of
\S\ref{sec:rotat-deviations}, there is no loss of generality assuming
$f\in\Homeo{k}(\T^2)$, for some $k\in\Z$, and $v=(0,1)$ in estimate
\eqref{eq:v-bounded-deviation}. So, the lift $\tilde f\colon\R^2\carr$
of $f$ commutes with the horizontal translation
$T_{(1,0)}\colon\R^2\carr$ and, consequently, induces an annulus
homeomorphism $\hat f\colon\A\carr$ characterized by the
semi-conjugacy equation
$\tilde\pi\circ \tilde f = \hat f\circ\tilde\pi$.

Now, we want to construct a surjective continuous map
$h\colon\T^2\to\T$ such that $h\circ f=T_{\rho}\circ h$, where $\rho$
is the irrational number given by \eqref{eq:v-bounded-deviation}. By
\cite[Theorem 1]{JagerPasseggiTorHomeosSemiConj} we know that there is
no loss of generality assuming that each fiber $h^{-1}(y)$, with
$y\in\T$, is an annular continuum as defined in
\S\ref{sec:dimension-two}. So, following
Jäger~\cite{JaegerLinearConsTorus}, we will define the semi-conjugacy
$h$ starting from its family of fibers.

More precisely, due to technical reasons we will start working on the
annulus $\A$ instead of $\T^2$ and we will construct a family of
continua $\{\ms{C}^s\subset\A : s\in\R\}$ such that each $\ms{C}^s$ is
an essential annular continuum in $\A$, they are well ordered
according to the index, \ie
\begin{equation}
  \label{eq:Cs-well-ordered}
  \ms{C}^s\subset U^-(\ms{C}^r), \quad\text{for all } s<r,
\end{equation}
where $U^-(\cdot)$ denotes the lower connected component of the
complement in $\A$ of the corresponding essential annular continuum as
defined in \S\ref{sec:dimension-two}; and the family satisfies the
following equivariant properties:
\begin{align}
  \label{eq:hat-T-on-C-s}
  \hat T_1\big(\mathscr{C}^s\big)&=\mathscr{C}^{s+1}, \\
  \label{eq:hat-f-on-C-s}
  \hat f\big(\mathscr{C}^s\big)&=\mathscr{C}^{s+\rho}, \quad\forall
                                 s\in\R.
\end{align}
Observe that, in particular, \eqref{eq:Cs-well-ordered} implies that
the continua are two-by-two disjoint, \ie
$\ms{C}^s\cap\ms{C}^r=\emptyset$, whenever $r\neq s$.

Then, we define the map $\hat h\colon\A\to\R$ by
\begin{equation}
  \label{eq:hat-h-definition}
  \hat h(z):=\inf\left\{s\in\R : z\in U^-\big(\ms{C}^s\big)\right\},
  \quad\forall z\in\A.
\end{equation}
By \eqref{eq:hat-T-on-C-s} we know that
$\hat h\circ\hat T_1=T_1\circ \hat h$ and, by \eqref{eq:hat-f-on-C-s},
$\hat h\circ \hat f=T_\rho\circ\hat h$. Thus, $\hat h$ is the lift of
a map $h\colon \T^2\to\T$ and it clearly holds
$h\circ f=T_\rho\circ h$. So, in order to show that $h$ is indeed a
semi-conjugacy, it just remains to prove that $\hat h$, and then $h$
as well, is continuous. This is a consequence of the irrationality of
$\rho$ and the fact that the essential annular continua
$\{\ms{C}^s : s\in\R\}$ are two-by-two disjoint. The reader can find
more details about the proof of the continuity of $\hat h$ given by
\eqref{eq:hat-h-definition}, assuming \eqref{eq:Cs-well-ordered},
\eqref{eq:hat-T-on-C-s} and \eqref{eq:hat-f-on-C-s}, in either
\cite[page 615]{JaegerLinearConsTorus} or \cite[Lemma
3.2]{JagerPasseggiTorHomeosSemiConj}.

\subsection{The construction of continua $\ms{C}^s $}
\label{sec:construction-Cs}

In order to finish the proof of Theorem A, it remains to construct the
family of essential annular continua $\{\ms{C}^s : s\in\R\}$
satisfying properties \eqref{eq:Cs-well-ordered},
\eqref{eq:hat-T-on-C-s} and \eqref{eq:hat-f-on-C-s}.

To do that, let $F\colon\T\times\A\carr$ be the $\rho$-centralized
skew-product induced by $\tilde f$ given by
\eqref{eq:rho-centr-skew-prod-def} and where $\rho$ is the irrational
number appearing in \eqref{eq:v-bounded-deviation}. Since $f$ exhibits
small wandering domains (Definition~\ref{def:small-wandering-domains})
and is periodic point free, by
Proposition~\ref{pro:small-wand-domains-implies-Omega-recurr} we know
that $f$ is $\Omega$-recurrent, and hence, by
Theorem~\ref{thm:Omega-f-vs-Omega-F} we know that $z\in\Omega(f)$ if
and only $(0,\hat z)\in\Omega(F)$, for every
$\hat z\in\hat\pi^{-1}(z)$.

The main new idea of the proof of Theorem A is given by the following
lemma, which could seem at first glance to be rather technical, but is
the core of the method:

\begin{lemma}
  \label{lem:Tau-F-invariant-construction}
  There exists an open bounded connected $F$-invariant subset
  $\Tt\subset\T\times\A$ such that
  \begin{equation}
    \label{eq:Tt-open-set-main-prop}
    \Big(\Tt\setminus\Gamma^{-s}\big(\overline{\Tt}\big)\Big)
    \cap\Omega(F)\neq\emptyset, \quad\forall s\in\R\setminus\{0\}.   
  \end{equation}
\end{lemma}

Before proving Lemma~\ref{lem:Tau-F-invariant-construction}, let us
see how the open set $\Tt$ can be used to construct our family of
annular continua $\{\ms{C}^s : s\in\R\}$.

First, observe that since $\Tt$ is bounded, connected and
$F$-invariant, and intersects $\Omega(F)$, it holds
$\Tt=\U_F(\Tt)$. So, by
Theorem~\ref{thm:separating-sets-orbits-of-open-sets} we know the set
$\A\setminus\Gamma^{-s}\big(\overline{\Tt}\big)_t$ has exactly two
unbounded connected components, for every $t\in\T$ and any
$s\in\R$. Hence, we can define the sets
\begin{displaymath}
  \Tt_t^{s,-}:=
  U^-\Big(\Gamma^{-s}\big(\overline{\Tt}\big)_t\Big)\subset \A, 
  \quad\forall t\in\T,\ \forall s\in\R, 
\end{displaymath}
\ie it is the lower unbounded connected component of the set
$\A\setminus\Gamma^{-s}\big(\overline{\Tt}\big)_t$. Then we write
\begin{equation}
  \label{eq:C-s-boundary-of-Tt-s-t}
  \mathscr{C}^s:=\partial_\A\big(\Tt_0^{s,-}\big)\subset\A, \quad
  \forall s\in\R. 
\end{equation}

Notice that each $\mathscr{C}^s$ is an essential annular
continuum. Moreover, since $F$ commutes with $\Gamma^{-1}$ and
$\Gamma^{-1}(t,z)=(t,z+1)$, for all $(t,z)\in\T\times\A$, then
\eqref{eq:hat-T-on-C-s} clearly holds.

In order to prove \eqref{eq:hat-f-on-C-s}, observe that, since
$\Gamma^{-s}(\Tt)$ is $F$-invariant and $F$ preserves the ends of the
space $\T\times\A$, we get
\begin{equation}
  \label{eq:F-image-of-Tau-minus}
  F\big(\{t\}\times
  \Tt_t^{s,-}\big)=\{T_\rho(t)\}\times\Tt_{T_\rho(t)}^{s,-},  
  \quad\forall t\in\T,\ \forall s\in\R, 
\end{equation}
and on the other hand,
\begin{equation}
  \label{eq:Gamma-minus-rho-image-Tau-minus}
  \Gamma^{-\rho}\big(\{T_\rho(t)\}\times\Tt_{T_\rho(t)}^{s,-}\big) =
  \{t\}\times \Tt_t^{s+\rho,-}, \quad\forall t\in\T,\ \forall s\in\R.  
\end{equation}
Putting together \eqref{eq:F-fundamental-prop},
\eqref{eq:F-image-of-Tau-minus} and
\eqref{eq:Gamma-minus-rho-image-Tau-minus} we get
\begin{displaymath}
  \hat f(\ms{C}^s) = \ms{C}^{s+\rho}, \quad\forall s\in\R,
\end{displaymath}
and \eqref{eq:hat-f-on-C-s} is proven.

So, it remains to prove that the continua $\{\ms{C}^s: s\in\R\}$ given
by \eqref{eq:C-s-boundary-of-Tt-s-t} satisfies condition
\eqref{eq:Cs-well-ordered}. To do that, it is important to notice that
is enough to show that
\begin{equation}
  \label{eq:Cs-disjoint-2-by-2}
  \ms{C}^s\cap\ms{C}^r=\emptyset, \quad\forall r,s\in\R,\ s\neq r. 
\end{equation}
In fact, one can easily check that condition
\eqref{eq:Cs-well-ordered} follows from \eqref{eq:Cs-disjoint-2-by-2}
and the combinatorics of these essential annular continua given by
conditions \eqref{eq:hat-T-on-C-s}, \eqref{eq:hat-f-on-C-s} and the
fact that
\begin{displaymath}
  \frac{\pr{2}\big(\hat f^n-id\big)}{n}\to\rho, \quad\text{as }
  n\to\infty
\end{displaymath}
uniformly on $\A$, which is a direct consequence of
\eqref{eq:v-bounded-deviation}.

So, condition \eqref{eq:Cs-disjoint-2-by-2} is the only remaining step
to prove that Theorem A is consequence of
Lemma~\ref{lem:Tau-F-invariant-construction}. To do that, let $r$ be
an arbitrary real number and $s\in (0,1)$. We will show that
$\ms{C}^r$ and $\ms{C}^{r+s}$ are disjoint. Then, recalling that
$\ms{C}^r$ and $\ms{C}^{r+s}$ are the boundary of the open set
$\Tt_0^{r,-}$ and $\Tt_0^{r+s,-}$, respectively, we consider the open
set
\begin{displaymath}
  \check{V}:=\Gamma^{-r}(\Tt)\setminus\Gamma^{-(r+s)}
  \big(\overline{\Tt}\big) = \Gamma^{-r}\Big(\Tt\setminus \Gamma^{-s} 
  \big(\overline{\Tt}\big)\Big)\subset\T\times\A.   
\end{displaymath}
Since $\Tt$ and $\Omega(F)$ are $F$-invariant, and by
Proposition~\ref{pro:F-Hs-commutes} $F$ and $\Gamma$ commute, we get
that $\check{V}$ is $F$-invariant, and by
Lemma~\ref{lem:Tau-F-invariant-construction}, it holds
$\check{V}\cap\Omega(F)\neq\emptyset$. Moreover, we claim there exists
a connected component $V\in\pi_0(\check{V})$ such that
\begin{equation}
  \label{eq:cc-V-inter-OmegaF-0-fiber}
  V\cap\Omega(F)\cap \big(\{0\}\times\A\big)\neq\emptyset.
\end{equation}
To prove this, first let $V'$ be any connected component of
$\check{V}$ that intersects $\Omega(F)$. If $(t,\hat z)$ is an
arbitrary point of $V'\cap\Omega(F)$, since $F$ and $\Gamma$ commute
and $V'$ is open, there is a positive number $\eps>0$ such that
\begin{displaymath}
  \Gamma^u(t,\hat z)\in V'\cap\Omega(F)\cap\{t+\pi(u)\}\times\A,
  \quad\forall u\in (-\eps,\eps).  
\end{displaymath}
Then, since $\rho$ is irrational, there is $n\in\Z$ such that
$0\in T^n_\rho\big(B_\eps(t)\big)\subset\T$; and hence,
$V:=F^n(V')\in\pi_0(\check{V})$ satisfies
\eqref{eq:cc-V-inter-OmegaF-0-fiber}.

Now, recalling we are assuming $f$ is not eventually annular, there
are two possible cases to be considered: either $f$ is a totally
irrational pseudo-rotation exhibiting uniformly bounded rotational
deviations (\ie estimate \eqref{eq:pseudo-rot-unif-bound-dev} holds);
or $f$ exhibits uniformly bounded vertical rotational deviations and
unbounded horizontal deviations (\ie estimate
\eqref{eq:horizontal-unbounded-deviations} holds).

In the first case, \ie when $f$ is a totally irrational
pseudo-rotation and \eqref{eq:pseudo-rot-unif-bound-dev} holds, the
disjointness of continua $\ms{C}^r$ and $\ms{C}^{r+s}$ easily follows
from
Lemma~\ref{lem:separating-sets-orbits-of-open-sets-tot-irrat-bound-dev}.
In fact, $V$ is a connected component of an $F$-invariant set and
intersects $\Omega(F)$. So, $V=\U_F(V)$ and by
Lemma~\ref{lem:separating-sets-orbits-of-open-sets-tot-irrat-bound-dev}
we know that $V_t$ is annular, for every $t\in\T$. This implies $V_t$
separates $\ms{C}^r$ and $\ms{C}^{r+s}$.

The second case, \ie when $f$ exhibits unbounded horizontal
deviations, will follow as consequence of
Lemma~\ref{lem:essential-annular-images-open-sets}, but it is a little
more involved.

Let $(0,\hat z)$ be any point belonging to $V\cap\Omega(F)$, and
$\delta>0$ be a sufficiently small number such that
\begin{equation}
  \label{eq:B-delta-block-in-V-def}
  B:= \big(B_\delta(0,\hat z)\big)^{\delta,0} =
  \bigcup_{\abs{u}<\delta} \Gamma^u\big(B_\delta(0,\hat z)\big)
  \subset V. 
\end{equation}
Without loss of generality we can assume $\delta<1/4$. We will show
that there exists $i\in\N$ and $q\in\Z$ such that
\begin{equation}
  \label{eq:Gamma-q-B-cup-Fi-B-separates-Cr-Crs}
  \big(\Gamma^{-q}(B)\cap F^i(B)\big)_0\subset\A
\end{equation}
is an essential set, and this will prove that $\ms{C}^r$ and
$\ms{C}^{r+s}$ are disjoint. In fact, we know that
$B_0\cap \ms{C}^r\subset V_0\cap\ms{C}^r=\emptyset$, and hence, by
\eqref{eq:hat-T-on-C-s},
$\big(\Gamma^{-q}(B)\big)_0\cap\hat T_q(\ms{C}^r) = \hat
T_q(B_0)\cap\ms{C}^{r+q}=\emptyset$, where $B_0$ denotes the fiber of
$B\subset\T\times\A$ over the point $0\in\T$. Analogously, one can
show that $\big(\Gamma^{-q}(B)\big)_0\cap \ms{C}^{r+s+q}=\emptyset$;
and since we are assuming $0<s<1$, we know that
$B_0\cap\ms{C}^{r+q}= B_0\cap\ms{C}^{r+s+q}=\emptyset$, for any
$q\in\Z$. So, condition \eqref{eq:Gamma-q-B-cup-Fi-B-separates-Cr-Crs}
clearly shows that $\ms{C}^{r+q}\cap\ms{C}^{r+s+q}=\emptyset$, and by
\eqref{eq:hat-T-on-C-s}, this holds if and only if
$\ms{C}^r\cap\ms{C}^{r+s}=\emptyset$.

So, let us show \eqref{eq:Gamma-q-B-cup-Fi-B-separates-Cr-Crs}
holds. By the minimality of $T_\rho\colon\T\carr$, we know there
exists a natural number $N$ such that for any $m\in\Z$, there is
$i\in\{m,m+1,\ldots,m+N\}$ such that
$T_\rho^i(0)\in B_{\delta/4}(0)\subset\T$.

On the other hand, if we write $z:=\hat\pi(\hat z)$, by
Theorem~\ref{thm:Omega-f-vs-Omega-F} we know that $z\in\Omega(f)$. Let
$\tilde z$ be any point in $\pi^{-1}(z)=\tilde\pi^{-1}(\hat
z)$. Applying Lemma~\ref{lem:essential-annular-images-open-sets} for
the open set $B_{\delta/4}(\tilde z)\subset\R^2$ and the natural
number $N$ given by the above condition, we know that there is
$m\in\N$ such that for every $j\in\{m,m+1,\ldots,m+N\}$, there are
$p_j,p_j',q_j\in\Z$ with $p_j\neq p_j'$ and such that
$\tilde f^j\big(B_{\delta/4}(\tilde z)\big)\cap
T_{(p_j,q_j)}\big(B_{\delta/4}(\tilde z)\big)\neq\emptyset$ and
$\tilde f^j\big(B_{\delta/4}(\tilde z)\big)\cap
T_{(p_j',q_j)}\big(B_{\delta/4}(\tilde z)\big)\neq\emptyset$; or in
other words, we get that the open set
$\hat T_{q_j}\Big(\tilde\pi\big(B_{\delta/4}(\tilde
z)\big)\Big)\cup\hat f^j\Big(\tilde\pi\big(B_{\delta/4}(\tilde
z)\big)\Big)$ is essential in $\A$, for every $j\in\{m,\ldots,m+N\}$.

Then, let $i$ be a natural number such that $i\in\{m,\ldots,m+N\}$ and
$T_\rho^i(0)\in B_{\delta/4}(0)$, and let us consider the open set
\begin{displaymath}
  \hat{B}^{(i)}:=\hat T_{q_i}\Big(\tilde\pi\big(B_{\delta/4}(\tilde
  z)\big)\Big)\cup\hat f^i\Big(\tilde\pi\big(B_{\delta/4}(\tilde
  z)\big)\Big) = \hat T_{q_i}\big(B_{\delta/4}(\hat z)\big)\cup \hat
  f^i\big(B_{\delta/4}(\hat z)\big) \subset\A,
\end{displaymath}
since we are assuming $\delta<1/4$. We know that $\hat{B}^{(i)}$ is an
essential set and putting together this with
Propositions~\ref{pro:F-Hs-commutes} and \ref{pro:F-image-of-r-block},
and \eqref{eq:B-delta-block-in-V-def}, we conclude that there is an
integer number $q$ such that the set given by
\eqref{eq:Gamma-q-B-cup-Fi-B-separates-Cr-Crs} is an essential set in
$\A$, as we wanted to prove.  Then, \eqref{eq:Cs-well-ordered} is
proven.

So, the last step of the proof of Theorem A is
\begin{proof}[Proof of Lemma~\ref{lem:Tau-F-invariant-construction}]
  The proof of this lemma is considerably simpler in the case that
  $\Omega(f)$ has nonempty interior (\eg when $f$ is non-wandering)
  than in the case where the wandering set is dense in $\T^2$.

  So, for the sake of simplicity of the exposition let us start
  proving the lemma under the assumption that the non-wandering set
  has nonempty interior. In such a case, let $V\subset\Omega(f)$ be an
  open, nonempty, lift-bounded and connected set.

  Since we are assuming $f$ is periodic point free and exhibits small
  wandering domains, by
  Proposition~\ref{pro:small-wand-domains-implies-Omega-recurr} we
  know that $f$ is $\Omega$-recurrent, and thus, by
  Theorem~\ref{thm:Omega-f-vs-Omega-F} it holds
  $\{0\}\times \hat\pi^{-1}(V)\subset \Omega(F)\subset\T\times\A$,
  where $\hat\pi\colon\A\to\T^2$ is the natural covering map given by
  \eqref{eq:hat-pi-def}.

  Let $\hat V$ be any connected component of
  $\hat\pi^{-1}(V)\subset\A$.  Since we are assuming $V$ is
  lift-bounded, $\hat V$ is bounded in $\A$ as well. Then, we define
  \begin{equation}
    \label{eq:Tau-def}
    \Tt:= \U_F\Big(\hat{V}^{\frac{1}{2},0}\Big),
  \end{equation}
  where $\hat{V}^{\frac{1}{2},0}$ denotes the $1/2$-block induced by
  $V$ and centered at $0\in\T$ given by \eqref{eq:V-r-block}, and
  $\U_F(\cdot)$ is given by \eqref{eq:U-eps-def}. By
  Theorem~\ref{thm:separating-sets-orbits-of-open-sets} we know that
  $\Tt$ is $F$-invariant, and by Proposition~\ref{pro:F-Hs-commutes}
  we know that $\Tt\subset\Omega(F)$.

  On the other hand, since $\hat V$ is bounded in $\A$, by
  Proposition~\ref{pro:F-orbit-bound-vert-dir} $\Tt$ is bounded in
  $\T\times\A$ as well. This implies that for every $s\neq 0$,
  $\Tt\setminus\Gamma^{-s}\big(\overline{\Tt}\big)$ is nonempty and,
  clearly
  $\Tt\setminus\Gamma^{-s}\big(\overline{\Tt}\big)\subset\Omega(F)$. Hence,
  Lemma~\ref{lem:Tau-F-invariant-construction} is proven under the
  additional hypothesis that $\Omega(f)$ has nonempty interior.

  So, from now on let us suppose the non-wandering set $\Omega(f)$ has
  empty interior.

  Let $z_0$ be any point of $\Omega(f)$ and $\hat z_0$ an arbitrary
  point in $\hat\pi^{-1}(z_0)\subset\A$. Now let us define the set
  $\hat V$ as the \emph{wandering completion} of the open unit ball
  $B_1(\hat z_0)\subset\A$, \ie $\hat V$ is given by
  \begin{equation}
    \label{eq:hat-V-open-set-def}
    \hat V:= B_1(\hat z_0)\cup \bigcup \left\{\hat\pi^{-1}(W) :
      W\in\pi_0\big(\T^2\setminus\Omega(f)\big), \hat\pi^{-1}(W)\cap
      B_1(\hat z_0)\neq\emptyset \right\}.
  \end{equation}
  Since $f$ exhibits small wandering domains, $\hat V$ is open,
  bounded and connected, and since $z_0\in\Omega(f)$, by
  Theorem~\ref{thm:Omega-f-vs-Omega-F} we know
  $(0,\hat z_0)\in\Omega(F)\cap\Tt$. Hence, the set
  $\Tt:= \U_F\Big(\hat{V}^{\frac{1}{2},0}\Big)$ is again
  $F$-invariant, open and bounded in $\T\times\A$. This implies
  $\Tt\setminus\Gamma^{-s}\big(\overline{\Tt}\big)\neq\emptyset$, for
  every $s\in\R\setminus\{0\}$.

  So, it just remain to prove that the set
  $\Tt\setminus\Gamma^{-s}\big(\overline{\Tt}\big)$ intersects
  $\Omega(F)$, for every $s\neq 0$. To do that, let us just consider
  the case where $s>0$; the other one is completely analogous.

  Now, observe that combining
  Proposition~\ref{pro:lift-bounded-iness-wandering-domains-ppf-homeos},
  Proposition~\ref{pro:F-image-of-r-block} and
  \eqref{eq:hat-V-open-set-def} we know that each connected component
  of the open set $\Tt\cap\W(F)$ is a $1/2$-block, each of them is a
  wandering set for $F$ and their $\Gamma$-orbits are two-by-two
  disjoint. More precisely, if $\check{W}_1,\check{W}_2,\ldots$ denote
  the (infinitely many) connected components of $\Tt\cap\W(F)$, then
  for each $n\geq 1$ there exists a unique connected component $W_n$
  of $\W(f)$ such that if we choose any connected component
  $\hat{W}_n$ of the open set $\hat\pi^{-1}(W_n)\subset\A$, there is a
  unique $r_n\in\R$ satisfying
  \begin{equation}
    \label{eq:Tt-WF-connected-components}
    \check{W}_n=\bigcup_{\abs{u}<\frac{1}{2}}
    \Gamma^{u+r_n}\big(\{0\}\times\hat{W}_n\big).
  \end{equation}

  On the other hand, by our definition \eqref{eq:hat-V-open-set-def}
  and Theorem~\ref{thm:Omega-f-vs-Omega-F} again, we see that for each
  $n\geq 1$, there exists at least a point
  $\hat z_n\in\partial_\A(\hat{W}_n)$ such that
  \begin{equation}
    \label{eq:z-hat-n-point-in-W-n-Omega}
    \Gamma^{r_n+u}(0,\hat
    z_n)\in\overline{\check{W}_n}\cap\Tt\cap\Omega(F), \quad\forall
    u\in(-1/2,1/2).  
  \end{equation}
  
  Reasoning by contradiction let us suppose that there exists $s>0$
  such that the open set
  $D:=\Tt\setminus\Gamma^{-s}\big(\overline{\Tt}\big)$ does not
  intersects $\Omega(F)$. This implies, invoking
  \eqref{eq:Tt-WF-connected-components}, that $D$ has infinitely many
  connected components and all of them can be enumerated as follows:
  \begin{equation}
    \label{eq:cc-D-description}
    D_n:=\bigcup_{\abs{u}<\frac{s}{2}}
    \Gamma^{r_n+\frac{1-s}{2}+u}\big(\{0\}\times\hat{W}_{n}\big),
    \quad\forall n\geq 1, 
  \end{equation}
  and $D=\bigsqcup_{n\geq 1} D_n$.
  
  Then, we will construct a subsequence $(D_{n_j})_{j\geq 1}$ of
  connected components of $D$ such that for every $M>0$ there is a
  $j\geq 1$ with $D_{n_j}\cap \T\times (-\infty,-M)\neq\emptyset$,
  contradicting the fact that $D$ was a bounded subset of
  $\T\times\A$.

  To do that, we will define the sequence $(D_{n_j})_{j\geq 1}$
  inductively. Let us start defining $n_1=1$. By
  \eqref{eq:z-hat-n-point-in-W-n-Omega} we know that
  \begin{displaymath}
    w_1:=\Gamma^{r_1+\frac{2-s}{4}}(0,\hat{z}_{1})\in
    \overline{D_1}\cap\Tt\cap\Omega(F),
  \end{displaymath}
  but since $D$ and $\Omega(F)$ are disjoint, we get that
  $w_1\not\in D_{n_1}=D_1$. So,
  $w_1\in\Gamma^{-s}\big(\overline{\Tt}\big)$. That means that there
  exists $n_2\in\N$ such that
  \begin{displaymath}
    d_{\T\times\A}\big(\Gamma^s(w_1), \check{W}_{n_2}\big) <
    \frac{1}{2^2},
  \end{displaymath}
  
  and we define
  \begin{displaymath}
    w_2:=\Gamma^{r_{n_2}+\frac{2-s}{4}}(0,\hat{z}_{n_2})\in
    \overline{D}_{n_2}\cap\Tt\cap\Omega(F).
  \end{displaymath}

  Inductively, supposing that $n_k\in\N$ and $w_k\in\T\times\A$ has
  been already defined, we choose $n_{k+1}\in\N$ as any natural number
  satisfying
  \begin{equation}
    \label{eq:n-k+1-def}
    d_{\T\times\A}\big(\Gamma^s(w_k), \check{W}_{n_{k+1}}\big) <
    \frac{1}{2^{k+1}},
  \end{equation}
  and then we define
  \begin{equation}
    \label{eq:w-k+1-def}
    w_{j+1}:=\Gamma^{r_{n_{k+1}}+\frac{2-s}{4}}(0,\hat{z}_{n_{k+1}})
    \in \overline{D}_{n_{k+1}}\cap\Tt\cap\Omega(F).
  \end{equation}

  Then, putting together \eqref{eq:Tt-WF-connected-components},
  \eqref{eq:z-hat-n-point-in-W-n-Omega}, \eqref{eq:cc-D-description},
  \eqref{eq:n-k+1-def} and \eqref{eq:w-k+1-def} we conclude that
  \begin{displaymath}
    \pr{3}(w_{k+1}) \leq \pr{3}(w_k) -s + \diam W_{n_{k+1}} +
    \frac{1}{2^k}, \quad\forall  k\geq 1.
  \end{displaymath}
  Iterating this last estimate one easily gets
  \begin{equation}
    \label{eq:last-cooridante-w-k+1-estimate}
    \begin{split}
      \pr{3}(w_{k+1})&\leq \pr{3}(w_1) + \sum_{j=1}^{k} \big(\diam
      W_{n_{j+1}} -s) + \sum_{j=1}^k \frac{1}{2^j} \\
      &< \pr{3}(w_1) + 1 + \sum_{j=1}^{k} \big(\diam W_{n_{j+1}} -s)
      \to -\infty,
    \end{split}
  \end{equation}
  as $k\to+\infty$, where the last limit follows from the small
  wandering domain hypothesis which implies $\diam W_{n_k}\to 0$, as
  $n_k\to+\infty$.

  Since $w_k\in\overline{D}_{n_k}\subset\overline{D}$, for every
  $k\in\N$, estimate \eqref{eq:last-cooridante-w-k+1-estimate} clearly
  contradicts the fact that $D$ was bounded in $\T\times\A$.
\end{proof}

\bibliographystyle{amsalpha} \bibliography{references}

\providecommand{\bysame}{\leavevmode\hbox to3em{\hrulefill}\thinspace}
\providecommand{\MR}{\relax\ifhmode\unskip\space\fi MR }
\providecommand{\MRhref}[2]{%
  \href{http://www.ams.org/mathscinet-getitem?mr=#1}{#2}
}
\providecommand{\href}[2]{#2}
\begin{thebibliography}{KPR18}

\bibitem[AZ05]{ZanataPropVertRotInt}
S.~Addas-Zanata, \emph{On properties of the vertical rotation interval for
  twist mappings}, Ergodic Theory and Dynamical Systems \textbf{25} (2005),
  no.~03, 641--660.

\bibitem[DM97]{DoeffMisiurewiczShearRotNum}
E.~Doeff and M.~Misiurewicz, \emph{Shear rotation numbers}, Nonlinearity
  \textbf{10} (1997), no.~6, 1755--1762.

\bibitem[Doe97]{DoeffRotMeasHomeos}
E.~Doeff, \emph{Rotation measures for homeomorphisms of the torus homotopic to
  a {D}ehn twist}, Ergodic Theory Dynam. Systems \textbf{17} (1997), no.~3,
  575--591.

\bibitem[Ede19]{EdekoEquicontFactorsFlows}
N.~Edeko, \emph{On equicontinuous factors of flows on locally path-connected
  compact spaces}, arXiv preprint arXiv:1904.12203, 2019.

\bibitem[Fat87]{FathiOrbCloBrouwer}
A.~Fathi, \emph{An orbit closing proof of {Brouwer{'}s} lemma on translation
  arcs}, Enseign. Math. (2) \textbf{33} (1987), no.~3-4, 315--322.

\bibitem[FM90]{FranksMisiure}
J.~Franks and M.~Misiurewicz, \emph{Rotation sets of toral flows}, Proc. Amer.
  Math. Soc. \textbf{109} (1990), no.~1, 243--249.

\bibitem[GH55]{GottschalkHedlundTopDyn}
W.~Gottschalk and G.~Hedlund, \emph{Topological dynamics}, vol.~36, American
  Mathematical Soc., 1955.

\bibitem[HJ17]{HauserJaegerMonotMaxEq}
T.~Hauser and T.~J\"ager, \emph{Monotonicity of maximal equicontinuous factors
  and an application to toral flows}, ArXiv preprint 1711.05672. To appear in
  Proceedings of the American Mathematical Society, 2017.

\bibitem[J{\"a}g09]{JaegerLinearConsTorus}
T.~J{\"a}ger, \emph{Linearization of conservative toral homeomorphisms},
  Inventiones Mathematicae \textbf{176} (2009), no.~3, 601--616.

\bibitem[JP15]{JagerPasseggiTorHomeosSemiConj}
T.~J{\"{a}}ger and A.~Passeggi, \emph{On torus homeomorphisms semiconjugate to
  irrational rotations}, Ergodic Theory and Dynamical Systems \textbf{35}
  (2015), no.~07, 2114--2137.

\bibitem[JT17]{JaegerTalIrratRotFact}
T.~J\"ager and F.~Tal, \emph{Irrational rotation factors for conservative torus
  homeomorphisms}, Ergodic Theory and Dynamical Systems \textbf{37} (2017),
  no.~5, 1537--1546.

\bibitem[KK09]{KocKorFoliations}
A.~Kocsard and A.~Koropecki, \emph{A mixing-like property and inexistence of
  invariant foliations for minimal diffeomorphisms of the 2-torus}, Proceedings
  of the American Mathematical Society \textbf{137} (2009), no.~10, 3379--3386.

\bibitem[Koc16]{KocMinimalHomeosNotPseudo}
A.~Kocsard, \emph{On the dynamics of minimal homeomorphisms of $\mathbb{T}^2$
  which are not pseudo-rotations}, ArXiv:1611.03784. To appear in Annales
  scientifiques de l{'}Ecole normale sup{\'{e}}rieure, 2016.

\bibitem[KPR18]{KocRotDevPerPFree}
A.~Kocsard and F.~Pereira-Rodrigues, \emph{Rotational deviations and invariant
  pseudo-foliations for periodic point free torus homeomorphisms}, Math. Z.
  \textbf{290} (2018), no.~3-4, 1223--1247.

\bibitem[KPS16]{KoropeckiPasseggiSambarino}
A.~Koropecki, A.~Passeggi, and M.~Sambarino, \emph{The {Franks-Misiurewicz}
  conjecture for extensions of irrational rotations}, preprint
  arXiv:1611.05498, 2016.

\bibitem[KT14]{KoroTalStricTor}
A.~Koropecki and F.~Tal, \emph{Strictly toral dynamics}, Inventiones
  Mathematicae \textbf{196} (2014), no.~2, 339--381.

\bibitem[MZ89]{MisiurewiczZiemian}
M.~Misiurewicz and K.~Ziemian, \emph{Rotation sets for maps of tori}, Journal
  of the London Mathematical Society \textbf{2} (1989), no.~3, 490--506.

\bibitem[Poi80]{Poincare1880memoire}
H.~Poincar{\'{e}}, \emph{M{\'{e}}moire sur les courbes d{\'{e}}finies par les
  {\'{e}}quations diff{\'{e}}rentielles {I{--}VI}, oeuvre {I}},
  Gauthier-Villar: Paris (1880), 375--422.

\bibitem[PS18]{PasseggiSambarinoDevFMConj}
A.~Passeggi and M.~Sambarino, \emph{Deviations in the {Franks--Misiurewicz}
  conjecture}, To appear in Ergodic Theory \& Dynamical Systems, 2018.

\bibitem[Tao09]{TaoPoincareLega}
T.~Tao, \emph{Poincar{\'e}'s legacies: pages from year two of a mathematical
  blog}, American Mathematical Soiety, 2009.

\bibitem[WZ18]{WangZhangRigidityPseudoRotNortonSull}
J.~Wang and Z.~Zhang, \emph{The rigidity of pseudo-rotations on the two-torus
  and a question of {N}orton-{S}ullivan}, Geom. Funct. Anal. \textbf{28}
  (2018), no.~5, 1487--1516.

\end{thebibliography}

\end{document}